\documentclass{amsart}
\usepackage{amssymb,graphicx}
\usepackage{mathrsfs}
\usepackage{color}
\usepackage{enumerate}
\usepackage{a4wide}
\usepackage[colorlinks=true,citecolor=blue,linkcolor=blue]{hyperref}

\newcommand{\Rl}{\mathbb{R}}
\newcommand{\Cplx}{\mathbb{C}}
\newcommand{\Itgr}{\mathbb{Z}}
\newcommand{\Ntrl}{\mathbb{N}}

\newcommand{\Bc}{\mathcal{B}}

\newcommand{\Kc}{\mathcal{K}}
\newcommand{\Lc}{\mathcal{L}}
\newcommand{\Mv}{\mathcal{M}}

\newcommand{\Sc}{\mathcal{S}}

\newcommand{\Tr}{\mathrm{Tr}}

\newcommand{\dom}{\mathrm{dom}}

\newcommand{\loc}{\mathrm{loc}}

\newcommand{\gf}{\mathfrak{g}}

\newcommand{\Heis}{\mathbb{H}}

\newcommand{\VN}{\mathrm{VN}}

\def\XXint#1#2#3{{\setbox0=\hbox{$#1{#2#3}{\int}$ }
\vcenter{\hbox{$#2#3$ }}\kern-.6\wd0}}

\numberwithin{equation}{section}

\newtheorem{theorem}{Theorem}[section]
\newtheorem{proposition}[theorem]{Proposition}
\newtheorem{corollary}[theorem]{Corollary}
\newtheorem{definition}[theorem]{Definition}
\newtheorem{lemma}[theorem]{Lemma}

\theoremstyle{remark}

\newtheorem{remark}[theorem]{Remark}

\title{Spectral estimates and asymptotics for stratified Lie groups}
\author[]{E. McDonald}
\address{School of Mathematics and Statistics, University of New South Wales, Kensington,  2052, Australia}
\email{edward.mcdonald@unsw.edu.au}
\author[]{F. Sukochev}
\address{School of Mathematics and Statistics, University of New South Wales, Kensington,  2052, Australia}
\email{f.sukochev@unsw.edu.au}

\author[]{D. Zanin}
\address{School of Mathematics and Statistics, University of New South Wales, Kensington,  2052, Australia}
\email{d.zanin@unsw.edu.au}

\date{\today}

\begin{document}

\maketitle

\begin{abstract}
We study Cwikel-type estimates for the singular values and Schatten $\mathcal{L}_p$-norms of compositions of multiplication and convolution operators acting on stratified Lie groups. This enables us to obtain novel spectral asymptotic formulas for certain operators derived from sub-Laplacians.   
\end{abstract}

\section{Introduction}
The so-called \emph{Cwikel estimates} concern the singular value and Schatten-class estimates for operators of the form
\[
    M_fg(-i\nabla):L_2(\Rl^d)\to L_2(\Rl^d).
\]
Here, $f \in L_{\infty}(\Rl^d)$ is a bounded function acting as a pointwise multiplier $(M_f\xi)(x) = f(x)\xi(x)$ and $g \in L_{\infty}(\Rl^d)$ acts as a Fourier multiplier. 
That is,
\[
    g(-i\nabla)\xi(x) = (2\pi)^{-d}\int_{\Rl^d} g(\omega)e^{i(x,\omega)}\widehat{\xi}(\omega)\,d\omega,\quad x \in \mathbb{R}^d.
\]  
We will call operators of the type $M_fg(-i\nabla)$ as \emph{product-convolution} operators, in view of the fact that $g(-i\nabla)$ acts as convolution by the inverse Fourier transform of $g.$
Product-convolution operators may be directly generalised to operators on the Hilbert space $L_2(G),$ where $G$ is a locally compact abelian group. Here, we consider
\[
    M_fT_g:L_2(G)\to L_2(G).
\]
In this case, $f \in L_{\infty}(G)$ acts on $L_2(G)$ by pointwise multiplication, and $g \in L_\infty(\widehat{G})$ acts on $L_2(G)$ by the Fourier multiplier
\[
    T_g\xi(\gamma) = \int_{\widehat{G}} g(\chi)\chi(\gamma)\widehat{\xi}(\chi)\,d\chi,\quad \gamma\in G.
\]
This may be further generalised to operators on $L_2(G),$ where $G$ is a locally compact group which need not be abelian, but merely unimodular.
The group von Neumann algebra $\VN(G)$ is the non-abelian generalisation of the algebra $L_{\infty}(\widehat{G}),$ and its representation as bounded linear operators on $L_2(G)$ generalises the representation of $L_{\infty}(\widehat{G})$ 
on $L_2(G)$ by Fourier multipliers.
In this case, the appropriate generalisation of product-convolution operators are of the form
\[
    M_fA:L_2(G)\to L_2(G).
\]
Again, $f \in L_{\infty}(G)$ acts on $L_2(G)$ by pointwise multiplication, and now $A$ belongs to the group von Neumann algebra $\VN(G).$ The group von Neumann algebra is generated by the set of all convolution operators
\[
    \lambda(h)\xi(\gamma) = \int_G h(\gamma\eta^{-1})\xi(\eta)\,d\eta,\quad h\in L_1(G)
\]
where $d\eta$ denotes the Haar measure on $G.$

A natural class of unimodular locally compact groups are the so-called stratified Lie groups. 

Let $G$ be a stratified $d$-dimensional real Lie group whose Lie algebra $\gf$ is stratified as $\bigoplus_{n=1}^\infty \gf_n$ and
has homogeneous dimension $d_{\hom}.$ Fix a basis set $\{X_1,\ldots,X_m\}$ for $\gf_1,$ realised as antisymmetric operators on $L_2(G),$ 
and denote by $\Delta$ the corresponding sub-Laplacian
\[
    \Delta = \sum_{k=1}^m X_k^2.
\]
The spectral theory of $\Delta$ has been studied by numerous authors \cite{CorwinGreenleaf1990,BLU-stratified-2007,Folland1975}. In particular, $\Delta$ is negative semidefinite and essentially self-adjoint on $L_2(G)$ with core $C^\infty_c(G).$ Equivalently, $\Delta$ may be defined via the quadratic form
\[
    Q(u) = \sum_{k=1}^m \int_{G} |X_ku(\gamma)|^2\,d\gamma,\quad u \in C^\infty_c(G).
\]
H\"ormander's criterion \cite[Theorem 22.2.1]{Hormander-III} implies
that $\Delta$ is subelliptic.

The operator $\Delta$ is affiliated with the group von Neumann algebra $\VN(G),$ and hence if $g$ is a bounded Borel function on the semiaxis $\Rl_+$ the operator $g(-\Delta)$
defined by functional calculus belongs to $\VN(G).$ It follows that the operator
\[
    M_fg(-\Delta):L_2(G)\to L_2(G)
\]
may be considered from the point of view of Cwikel's estimates.

We can illustrate our discussion with the case of the prototypical stratified Lie group: the three-dimensional Heisenberg group $\Heis^1.$ The underlying Lie algebra of this Lie group can be realised as the linear span of the following three vector fields
on $\Rl^3:$
\[
X = \partial_x - 2M_y\partial_t,\quad  Y = \partial_y+2M_x\partial_t,\quad T = \partial_t.
\]
Here, $x,y$ and $t$ are the coordinates of $\Rl^3.$ In this case,
\[
\Delta = X^2+Y^2 = \partial_x^2+\partial_y^2+4M_{x^2+y^2}\partial_t^2+4(M_x\partial_y-M_y\partial_x)\partial_t.
\]
Another example arising from a stratified Lie group is the operator on $L_2(\Rl^{2+N})$, where $N\geq 1,$ with coordinates $(t,s,x_1,x_2,\ldots,x_N)$ given by
\[
    \Delta = \partial_t^2+(\partial_s+M_t \partial_{x_1}+M_{t^2}\partial_{x_2}+\cdots+M_{t^N}\partial_{x_N})^2.
\]
This is a Bony-type sub-Laplacian corresponding to a particular stratified Lie group structure of $\Rl^{2+N},$ see \cite[Section 4.3.3]{BLU-stratified-2007}. 
Another example is the operator on $L_2(\Rl^N)$ (with coordinates $(x_1,\ldots,x_N)$ 
given by
\[
    \Delta = \partial_{x_2}^2+ (\partial_{x_1}+M_{x_2}\partial_{x_3}+M_{x_3}\partial_{x_4}+\cdots+M_{x_{N-1}}\partial_{x_N})^2.
\]  
For details on the Lie group underlying this example, see \cite[Section 4.3.5]{BLU-stratified-2007}. Further examples are listed in \cite[Chapter 4]{BLU-stratified-2007}.


The following is our main result, which in $p>2$ is a form of specific Cwikel estimate for stratified Lie groups and $p\leq 2$ are related to estimates due to Birman and Solomyak. All unexplained notations and terminology will be introduced in the next section.
\begin{theorem}\label{main_nontrivial_cwikel_theorem}
Let $G$ be a stratified Lie group with stratification $\gf = \bigoplus_{n=1}^\infty \gf_n,$ homogeneous dimension $d_{\hom} = \sum_{n=1}^\infty n\cdot \mathrm{dim}(\gf_n)$
and a fixed sub-Laplacian $\Delta = \sum_{j=1}^m X_j^2,$ where $\{X_j\}_{j=1}^m$ is a basis for $\gf_1.$ 
\begin{enumerate}[{\rm (i)}]
\item\label{mncta} if $p>2,$ then
$$\|M_f(-\Delta)^{-\frac{d_{\hom}}{2p}}\|_{p,\infty}\leq c_p\|f\|_{L_p(G)}$$
\item\label{mnctb} if $p<2$ and $q>2,$ then
$$\|M_f(1-\Delta)^{-\frac{d_{\hom}}{2p}}\|_{p,\infty}\leq c_{p,q}\|f\|_{\ell_p(L_q)(G)}.$$ 
\item\label{mnctc} if $p=2$ and $q>2,$ then
$$\|M_f(1-\Delta)^{-\frac{d_{\hom}}{2p}}\|_{p,\infty}\leq c_q\|f\|_{\ell_{2,\log}(L_q)(G)}.$$ 
\end{enumerate}
\end{theorem}

Having Theorem \ref{main_nontrivial_cwikel_theorem} at hands and using H\"older's inequality, we immediately obtain the following corollary.

\begin{corollary}\label{main_nontrivial_cwikel_corollary} Let $G$ be a stratified Lie group with stratification $\gf = \bigoplus_{n=1}^\infty \gf_n,$ homogeneous dimension $d_{\hom} = \sum_{n=1}^\infty n\cdot \mathrm{dim}(\gf_n)$
and a fixed sub-Laplacian $\Delta = \sum_{j=1}^m X_j^2,$ where $\{X_j\}_{j=1}^m$ is a basis for $\gf_1.$ 
\begin{enumerate}[{\rm (i)}]
\item if $p>1,$ then
$$\|(-\Delta)^{-\frac{d_{\hom}}{4p}}M_f(-\Delta)^{-\frac{d_{\hom}}{4p}}\|_{p,\infty}\leq c_p\|f\|_{L_p(G)}.$$
\item if $p<1$ and $q>1,$ then
$$\|(1-\Delta)^{-\frac{d_{\hom}}{4p}}M_f(1-\Delta)^{-\frac{d_{\hom}}{4p}}\|_{p,\infty}\leq c_{p,q}\|f\|_{\ell_p(L_q)(G)}.$$ 
\item if $p=1$ and $q>1,$ then
$$\|(1-\Delta)^{-\frac{d_{\hom}}{4p}}M_f(1-\Delta)^{-\frac{d_{\hom}}{4p}}\|_{p,\infty}\leq c_q\|f\|_{\ell_{1,\log}(L_q)(G)}.$$ 
\end{enumerate}	
\end{corollary}

Of course, a similar result holds for Schatten ideals.

\begin{theorem}\label{cwikel schatten_theorem}
Let $G$ be a stratified Lie group with stratification $\gf = \bigoplus_{n=1}^\infty \gf_n,$ homogeneous dimension $d_{\hom} = \sum_{n=1}^\infty n\cdot \mathrm{dim}(\gf_n)$ and a fixed sub-Laplacian $\Delta = \sum_{j=1}^m X_j^2,$ where $\{X_j\}_{j=1}^m$ is a basis for $\gf_1.$
\begin{enumerate}[{\rm (i)}]
\item if $p>2$ and $r>\frac{d_{\hom}}{p},$ then
$$\|M_f(-\Delta)^{-\frac{r}{2}}\|_p\leq c_{p,r}\|f\|_{L_p(G)}.$$
\item if $p=2$ and $r>\frac{d_{\hom}}{p},$ then
$$\|M_f(1-\Delta)^{-\frac{r}{2}}\|_p=c_{p,r}\|f\|_{L_p(G)}.$$
\item if $p<2,$ $r>\frac{d_{\hom}}{p}$and $q>2,$ then
$$\|M_f(1-\Delta)^{-\frac{r}{2}}\|_p\leq c_{p,q,r}\|f\|_{\ell_p(L_q)(G)}.$$ 
\end{enumerate}
\end{theorem}

Here, the notation $\|\cdot\|_{p}$ and $\|\cdot\|_{p,\infty}$ denote the $p$-Schatten and weak $p$-Schatten quasi-norms respectively. The notation $\ell_{\frac{d_{\hom}}{\beta}}(L_q)(G)$ refers to a particular function
space on $G$ defined carefully below. This is a generalisation of the mixed $\ell_p(L_q)(\mathbb{R}^d)$ spaces of Birman and Solomyak, see \cite[Chapter 4]{Simon-trace-ideals-2005}.

In the particular case when $G=\Rl^d,$ the sub-Laplacian reduces to the classical Laplacian $\Delta_{\mathbb{R}^d} = \sum_{j=1}^d \partial_j^2.$ The estimates in Theorem \ref{main_nontrivial_cwikel_theorem}
are not as strong as the best known estimates for product-convolution operators in $\Rl^d$ see e.g. \cite[Theorem 4.6]{Simon-trace-ideals-2005}, \cite[Corollary 4.6]{LeSZ-cwikel}. When specialised to $G = \Rl^d,$ the $\Lc_{2,\infty}$ estimate in Theorem \ref{main_nontrivial_cwikel_theorem} is not more or less general than \cite[Theorem 1.3]{LeSZ-cwikel}, as it applies to a wider class of multipliers $f$ but a smaller class of functions of $-\Delta.$ Nevertheless it is less
general than \cite[Theorem 1.3]{SZ-critical}. The stronger estimates that have been obtained in the Euclidean (or toroidal) case are made possible due to the special properties of the Fourier basis in $L_2(\mathbb{T}^d)$ and those proofs do not appear to be directly applicable in the general stratified case.

Schatten and weak-Schatten estimates for product-convolution operators have numerous applications in mathematical physics, some of which include:
\begin{itemize}
    \item{} Schatten ideal estimates such as those in Theorem \ref{cwikel schatten_theorem} readily imply certain results in scattering theory. For example, Theorem \ref{main_nontrivial_cwikel_theorem} implies the existence and completeness of the wave operators $\Omega_{\pm}(-\Delta+M_V,-\Delta)$ for certain classes of potentials $V.$
We discuss this in Section \ref{scattering_section}. 
    \item{} Weak Schatten ideal estimates such as in Corollary \ref{main_nontrivial_cwikel_corollary} are used in proving the Cwikel--Lieb--Rozenblum inequality, which estimates the number of bound states of a Schr\"odinger operator.
    \item{} Operators of the form $(1-\Delta)^{-\frac{d}{4}}M_f(1-\Delta)^{-\frac{d}{4}}$ are used in the formulation of Connes integration formula \cite{Connes-Action}. We prove a strong version of Connes' integration formula for stratified Lie groups below, see Remark \ref{cif_remark}.
    \item{} More generally, Schatten and weak-Schatten estimates for product-convolution operators are used to prove semiclassical Weyl laws for Schr\"odinger operators under minimal smoothness assumptions, see Corollary \ref{semi-classical corollary} below.
\end{itemize}

Related to these estimates we have spectral asymptotics. In the following theorem, $\mu$ denotes the singular value function. In particular, the sequence $\{\mu(n,T)\}_{n=0}^\infty$ is the sequence of singular values of a compact operator $T.$ We give a precise definition of $\mu$ in the next section.
\begin{theorem}\label{main_asymptotic_formula}
Let $G$ be a non-abelian stratified Lie group with stratification $\gf = \bigoplus_{n=1}^\infty \gf_n,$ homogeneous dimension $d_{\hom} = \sum_{n=1}^\infty n\cdot \mathrm{dim}(\gf_n)$
and a fixed sub-Laplacian $\Delta = \sum_{j=1}^m X_j^2,$ where $\{X_j\}_{j=1}^m$ is a basis for $\gf_1.$ Let $k\in\mathbb{N}$ and let $p=\frac{d_{\hom}}{k}.$ If one of the following conditions holds
\begin{enumerate}[{\rm (i)}]
\item{} $p>1$ and $0\leq f \in L_p(G);$ 
\item{} $p<1$ and $0\leq f\in \ell_p(L_q)(G)$ for some $q>1;$
\item{} $p=1$ and $0\leq f\in \ell_{1,\log}(L_q)(G)$ for some $q>1;$
\end{enumerate}
then there exists the limit
$$\lim_{t\to\infty} t\mu(t,(1-\Delta)^{-\frac{k}{4}}M_f(1-\Delta)^{-\frac{k}{4}})^p=c_G\int_G f^p.$$
Moreover, if $k<d_{\hom},$ then we also have the existence of the limit
$$\lim_{t\to\infty} t\mu(t,(-\Delta)^{-\frac{k}{4}}M_f(-\Delta)^{-\frac{k}{4}})^p=c_G\int_G f^p.$$
Here, the constant $c_G>0$ depends on the stratification and also on the particular choice of the basis in $\gf_1.$
\end{theorem}

\begin{remark}\label{cif_remark} For $k=d_{\hom},$ Theorem \ref{main_asymptotic_formula} implies a form of Connes' Integration Formula \cite{Connes-Action} for stratified Lie groups. For all continuous normalised traces $\varphi$ on $\Lc_{1,\infty},$ the asymptotic formula given in Theorem \ref{main_asymptotic_formula} yields
$$\varphi((1-\Delta)^{-\frac{d_{\hom}}{4}}M_f(1-\Delta)^{-\frac{d_{\hom}}{4}}) = c_G\int_G f,\quad f \in \ell_{1,\log}(L_q)(G),\; q>1.$$
\end{remark}

Theorem \ref{main_asymptotic_formula} is proved under the perhaps unnatural assumption that $k>0$ is an integer, we conjecture that this assumption 
is unnecessary. Similarly, it is reasonable to expect that the result admits nontrivial extension to non-positive $f,$ or matrix-valued $f.$

In the abelian case, that is when $G = \Rl^d$ and $\Delta = \sum_{j=1}^d \partial_j^2$ is the usual Laplace operator, asymptotics of this nature are
due to Birman and Solomyak see e.g. \cite{BirmanSolomyakWeaklyPolar1970,BirmanSolomyakNonSmooth1970}. The connection between Birman--Solomyak's asymptotic formulae
and Connes' integration formula was first publicized by Rozenblum \cite{Rozenblum2021arxiv}. 

The main difference to the elliptic case is that $d_{\hom}$ is typically larger than the dimension of $G$, and so Theorem \ref{main_asymptotic_formula} could be considered a non-Weylian asymptotic formula. 

A very closely related theme is the spectral theory of Heisenberg elliptic operators on Heisenberg manfolds, of which the sub-Laplacian of the Heisenberg group $\Heis^1$ is a very special case. The theory of pseudodifferential operators adapted to a Heisenberg manifold was originally developed by Beals-Greiner \cite{BealsGreiner1982,BealsGreiner1988} and Taylor \cite{Taylor1984}. Further developments include the ``intrinsic" calculus of Ponge \cite{PongeAMS2008}. The advantage of a general
pseudodifferential calculus is in its versatility. For example, using an intrinstic Heisenberg calculus Ponge obtained spectral asymptotics for a very broad class of hypoelliptic pseudodifferential operators on Heisenberg manifolds \cite[Chapter 6]{PongeAMS2008}. For some more recent work detailing the explicit relation to noncommutative geometry, see \cite{GimperleinGoffeng2017}. 

The present work is orthogonal to the pseudodifferential calculi of Beals-Greiner and Taylor in the sense we aim for far less generality in terms of operators and manifolds, but the nature of the underlying sub-Riemannian structure is more general. Our contribution is more in line with the spectral theory of differential operators associated to manifolds with a Carnot structure (in the sense of \cite{ChoiPonge2019}), or some other sub-Riemannian structure. As early as 1976, M\'etivier \cite{Metivier1976} proved pointwise asymptotics for the spectral function in this setting. Research at this level of generality continues to the present day, for example in the recent work of Colin de Verdi\`{e}re, Hillairet and Tr\'{e}lat \cite{CdVHT2018,CdVHTarxiv}.

The constant $c_G$ 
is unspecified in the statement of Theorem \ref{main_asymptotic_formula}, but may be computed in terms of the convolution kernel $h_t$ of the semigroup $e^{t\Delta}$ generated by $\Delta.$
The relation is that
\[
    c_G = \frac{1}{\Gamma(\frac{d_{\hom}}{2}+1)}t^{\frac{d_{\hom}}{2}}h_t(1),\quad t > 0.
\]
Here, $1\in G$ denotes the identity element of $G.$
The value of $t$ is arbitrary. For the case where $G = \Rl^d$ and $\Delta$ is the Laplacian, the classical formula $h_t(1) = (4\pi t)^{-\frac{d}{2}}$ recovers the value
\[
    c_{\Rl^d} = \frac{\mathrm{Vol}(\mathbb{S}^{d-1})}{d(2\pi)^d}.
\]
If $G = \Heis^n$ and $\Delta$ is the standard sub-Laplacian, we instead have
\[
    c_{\Heis^n} = \frac{1}{(n+1)!}(4\pi)^{-n-1}\int_{0}^\infty \left(\frac{\lambda}{\sinh(\lambda)}\right)^n\,d\lambda.
\]
This follows from a known explicit form of the heat kernel on $\Heis^n,$ \cite[Equation (1.73)]{Gaveau1977}. 
In general, the heat kernel $h_t$ can be given in terms of the unitary dual of $G,$ see e.g. \cite[Theorem 26(iii)]{ABGR2009}.

Via an argument using the Birman-Schwinger principle, Theorem \ref{main_asymptotic_formula} implies the following semiclassical Weyl law:
\begin{corollary}\label{semi-classical corollary}
    Let $G\neq\mathbb{H}^1$ be a stratified Lie group with stratification $\gf = \bigoplus_{n=1}^\infty \gf_n,$ homogeneous dimension $d_{\hom} = \sum_{n=1}^\infty n\cdot \mathrm{dim}(\gf_n)$
    and a fixed sub-Laplacian $\Delta = \sum_{j=1}^m X_j^2,$ where $\{X_j\}_{j=1}^m$ is a basis for $\gf_1.$

    Assume that $V\in L_{\frac{d_{\hom}}{2}}(G)$ is real-valued. For $h>0,$ the operator $-h^2\Delta\dot{+}M_V$ can be defined in the sense of quadratic forms.
    There exists a constant $c_{G}>0$ such that
    \[
        \lim_{h\to 0} h^{d_{\hom}}\Tr(\chi_{(-\infty,0)}(-h^2\Delta\dot{+}M_V)) = c_{G}\int_{G} V_-^{\frac{d_{\hom}}{2}}.
    \]
    Here, $V_- = \frac{1}{2}(|V|-V)$ is the negative part of $V.$
\end{corollary}
The constant $c_G$ in the above corollary is the same as in Theorem \ref{main_asymptotic_formula}. The assumption that $V_+ \in L_{\frac{d_{\hom}}{2}}(G)$ is likely to be unnecessarily
strong, and is only made here for convenience and to simplify the exposition.
Semiclassical Weyl estimates for the classical Laplacian are well-known see e.g. \cite{Rozenbluym1972}. As far as the authors are aware, semiclassical asymptotics
of these nature for operators such as $\Delta$ on $\Heis^n,$ $n>1,$ are novel.

We wish to extend our gratitude to Professor G.~Rozenblum for helpful comments and discussion relating to this paper, and to J.~Li and Z.~Fan for introducing us to this topic.

\section{Notations and conventions}
A real Lie algebra $\gf$ is said to be {graded} if $\gf$ can be decomposed as a direct sum $\gf = \bigoplus_{n=1}^\infty \gf_n,$ where $[\gf_n,\gf_m]\subset \gf_{n+m},$
and further is said to be {stratified} if $\gf_1$ generates $\gf$ as a Lie algebra \cite[Section 3.1]{FischerRuzhansky2016}. The homogeneous dimension $d_{\hom}$ of a graded
Lie algebra is defined to be
\[
    d_{\hom} := \sum_{n=1}^\infty n \cdot \mathrm{dim}(\gf_n).
\]    
We exclusively consider finite dimensional Lie algebras, and hence a graded Lie algebra is nilpotent and the above sum is finite.
A connected, simply connected Lie group $G$ is called stratified if its Lie algebra $\gf$ is stratified. A stratified Lie group
$G$ is diffeomorphic to $\gf$, via the exponential mapping $\exp$ \cite[Proposition 1.2]{FollandStein1982}, and the Haar measure of $G$ is identical
to the pushforward of the Lebesgue measure on $\gf$ under $\exp.$ We will throughout this paper identify $G$ with $\mathfrak{g}.$ The Lie algebra $\gf$ may also be identified with the set
of all vector fields on $G$ commuting with the action of $G$ on itself by right translation. 
 

Let $G$ be a stratified $d$-dimensional Lie group, with stratification $\gf = \bigoplus_{n=1}^\infty \gf_n,$ and with a particular fixed sub-Laplacian $\Delta = \sum_{k=1}^m X_k^2$
corresponding to a basis $\{X_1,\ldots,X_m\}$ of $\gf_1.$ Since the elements $\{X_1,\ldots,X_m\}$ are represented here as vector fields on $G$ commuting with right translation, it follows that $\Delta$
commutes with right translations.

As we identify $G$ with $\gf,$ each $X_k$
may be considered as a vector field on $\Rl^d$. The nilpotency of $G$ implies that the coefficients of each $X_k$ are polynomials in the coordinate variables. Similarly, $\Delta$ is a second order
linear differential operator on $\Rl^d$ with polynomial coefficients. The Haar measure on $G$ is identical to the Lebesgue measure on $\gf,$ and when writing $L_p$-spaces
such as $L_p(G)$ these are always defined with the Haar measure. We omit the measure when writing integrals of $f \in L_1(G).$ That is,
\[
    \int_G f := \int_{\gf} f(x)\,dx.
\]

Throughout we will denote constants depending on $G$ by $c_G,$ or other parameters $p,q,\ldots$ by $c_{p,q,G},$ etc. The value of the constant may change from line to line.
Estimates stated as depending on $G$ may also depend on the choice of generating set $\{X_1,\ldots,X_m\},$ although in general the same stratified Lie group can have more than one generating set. Throughout this paper, $G$ is a fixed
stratified Lie group with fixed generating set $\{X_1,\ldots,X_m\}.$

It follows from the characterisation of Sobolev spaces on stratified Lie groups \cite[Corollary 4.13]{Folland1975} that
\begin{equation}\label{sobolev_equivalence}
(1-\Delta)^{\frac{\alpha}{2}}X_l(1-\Delta)^{-\frac{\beta}{2}} \in \Bc(L_2(G)),\quad \alpha,\beta\in\mathbb{R},\quad \alpha+1\leq\beta.
\end{equation}

For a Hilbert space $H$, denote by $\Bc(H)$ the algebra of all bounded linear endomorphisms of $H.$ The ideal of all compact
linear maps is denoted $\Kc(H).$ The operator norm on $\Bc(H)$ is denoted $\|\cdot\|_{\infty}.$ Given $T \in \Kc(H),$ the singular value sequence $\mu(T):=\{\mu(n,T)\}_{n=0}^\infty$ is defined in terms of the singular value function
\[
    \mu(t,T) := \inf\{\|T-R\|_{\infty}\;:\; \mathrm{rank}(R)\leq t\},\quad t\geq 0.
\]
Equivalently, $\mu(n,T)$ is the $(n+1)$-st eigenvalue of $|T|$, arranged in decreasing order with multiplicities.

For $0 < p< \infty,$ the Schatten ideal $\Lc_p(H)$ consists of all compact linear operators $T$ such that
\[
    \|T\|_p := \left(\sum_{n=0}^\infty \mu(n,T)^p\right)^{\frac1p}<\infty. 
\]
We also make extensive use of the weak Schatten ideals, defined with the weak Schatten (quasi)-norms
\[
    \|T\|_{p,\infty} := \sup_{n\geq 0} (n+1)^{\frac{1}{p}}\mu(n,T).
\]

We have the H\"older-type inequalities
\[
    \|TS\|_r \leq \|T\|_p\|S\|_q,\quad \|TS\|_{r,\infty} \leq c_{p,q,r}\|T\|_{p,\infty}\|S\|_{q,\infty},\quad \frac{1}{r} = \frac{1}{p}+\frac{1}{q}.
\]
In this paper, the choice of Hilbert space is always clear from context; hence we abbreviate $\Lc_p(H)$ as $\Lc_p,$ similarly $\Lc_{p,\infty}(H)$ as $\Lc_{p,\infty}$ and so on.

For $0 < p < \infty,$ denote by $(\Lc_{p,\infty})_0$ the separable part of $\Lc_{p,\infty},$ defined
as the closure of the space of finite rank operators on $\Lc_{p,\infty},$ or equivalently the subset of $T \in \Lc_{p,\infty}$ such that
\[
    \lim_{n\to\infty} (n+1)^{\frac{1}{p}}\mu(n,T) = 0.
\]

An important feature of $(\Lc_{p,\infty})_0$ we use repeatedly is the following:
\begin{proposition}\label{elementary_perturbation} Let $T_1,T_2\in \Lc_{p,\infty}$ be such that $T_1-T_2\in(\Lc_{p,\infty})_0.$ If there exists the limit
$$\lim_{t\to\infty}t^{\frac1p}\mu(t,T_1) = c,$$
then there exists the limit
$$\lim_{t\to\infty}t^{\frac1p}\mu(t,T_2) = c.$$
\end{proposition}
This elementary fact is due to K.~Fan, see e.g. \cite[Chapter II, Theorem 4.1]{GohbergKrein}. Proposition \ref{elementary_perturbation}  a special case of stronger asymptotic stability results, see e.g. \cite[Section 11]{RozenblumShubinSolomyak1989}. In fact, a generalisation
of Proposition \ref{elementary_perturbation} due to Birman and Solomyak is the following:
\begin{proposition}\label{advanced_perturbation}\cite[Lemma 11.2]{RozenblumShubinSolomyak1989} Let $p>0$ and let $\{T_m\}_{m\geq0}\subset\Lc_{p,\infty}$ be such that for every $m\geq0$ there exists a limit
$$\lim_{t\to\infty}t^{\frac1p}\mu(t,T_m)=c_m.$$
If $T_m\to T$ in $\mathcal{L}_{p,\infty},$ then the following limits exist and are equal
$$\lim_{t\to\infty}t^{\frac1p}\mu(t,T)=\lim_{m\to\infty}c_m.$$ 
\end{proposition}


The operator trace on $\Bc(H)$ is denoted $\Tr.$ Note that
\[
    \mu(t,T) = \inf\{s\geq 0\;:\; \Tr(\chi_{(s,\infty)}(|T|))\leq t\},\quad T \in \Kc.
\]
From this relation it follows that the asymptotics of $\mu(t,T)$ as $t\to\infty$ are equivalent to those of $\Tr(\chi_{(s,\infty)}(T))$ as $s\downarrow 0.$ To be precise,
we have the following
\begin{lemma}\label{trivial spectral lemma} 
    Let $p>0$ and let $0\leq A\in\mathcal{L}_{p,\infty}.$ We have
    $$
    \limsup_{h\downarrow0}h^{p}\Tr(\chi_{(h,\infty)}(A))=\limsup_{h\downarrow0}h^{p}\Tr(\chi_{[h,\infty)}(A))=\limsup_{t\to\infty}t\mu(t,A)^{p}.
    $$
    Similarly,
    $$
    \liminf_{h\downarrow0}h^{p}\Tr(\chi_{(h,\infty)}(A))=\liminf_{h\downarrow0}h^{p}\Tr(\chi_{[h,\infty)}(A))=\liminf_{t\to\infty}t\mu(t,A)^{p}.
    $$
    Moreover,
    $$
    \lim_{h\downarrow0}h^{p}\Tr(\chi_{(h,\infty)}(A))=\lim_{h\downarrow0}h^{p}\Tr(\chi_{[h,\infty)}(A))=\lim_{t\to\infty}t\mu(t,A)^{p}
    $$
    if any of the limits exist.
\end{lemma}
For further details on trace ideals and the singular value function see \cite{LSZ,GohbergKrein,Simon-trace-ideals-2005}.

The space $\Sc(G)$ of Schwartz class functions on $G$ is defined via the identification
of $G$ with $\Rl^d,$ and $\Sc(G)$ is the usual Schwartz space $\Sc(\Rl^d)$ with its canonical Fr\'echet topology. The space of tempered distributions $\Sc'(G)$ is the topological dual of $\Sc(G).$

%

For $n\in\Ntrl$ define the Sobolev space $W^n_p(G)$ as the space of tempered distributions $f$ on $G$ such that
$$\|f\|_{W^n_p} := \|f\|_{L_p(G)}+\sum_{k=1}^n\sum_{1\leq j_1,\ldots,j_k\leq m} \|X_{j_1}X_{j_2}\cdots X_{j_k}f\|_{L_p(G)}.$$

For $s\in\mathbb{R},$ define Bessel-potential Sobolev space $H^s(G)$ as the space of tempered distributions $f$ on $G$ such that
$$\|f\|_{H^s}=\|(1-\Delta)^{\frac{s}{2}}f\|_{L_2(G)}<\infty.$$
By Corollary 4.13 in \cite{Folland1975}, we have $W^n_2=H^n$ with equivalent norms for every $n\in\mathbb{N}.$

\section{Singular value estimates for product-convolution operators in Group von Neumann algebras}\label{elementary_cwikel_section}
In this section we discuss Cwikel type estimates for $\Lc_p$, $2\leq p < \infty$ and $\Lc_{p,\infty}$, $2 < p < \infty.$ These are the simplest cases that are well-covered
by the abstract theory developed in \cite{LeSZ-cwikel}. 

To take the most advantage of that theory, for this section and only this section we will work in a much more general context than the rest of the paper. Here,
$G$ denotes an arbitrary unimodular locally compact group. Recall that a locally compact group is said to be unimodular if its left-invariant Haar measure coincides
with its right-invariant Haar measure, or equivalently if its left-invariant Haar measure $\nu$ satisfies
\[
    \nu(X) = \nu(X^{-1})
\]
for all measurable subsets $X$ of $G.$ All nilpotent Lie groups are unimodular,
as the Lebesgue measure is bivariant \cite[Proposition 1.2(c)]{FollandStein1982}, and so the results of this section apply to the stratified Lie groups considered in the remainder of the paper.

We write $L_2(G)$ for $L_2(G,\nu)$, and $L_{\infty}(G)$ for $L_\infty(G,\nu)$ and so forth.

Given $f \in C_c(G)$, write $\lambda(f)$ for the operator on $L_2(G)$ given by
\[
    (\lambda(f)u)(\gamma) = \int_{G} f(\gamma \eta^{-1})u(\eta)\,d\nu(\eta),\quad u \in L_2(G),\; \gamma \in G.
\]
Observe that $\lambda(f)\lambda(g)=\lambda(f\ast g),$ where $f\ast g \in C_c(G)$ is the convolution
\[
        (f\ast g)(\gamma) := \int_{G} f(\gamma \eta^{-1})g(\eta)\,d\nu(\eta),\quad f,g \in C_c(G),\;\gamma \in G.
\]
With this notation, $\lambda(f)u := f\ast u$ for $u\in L_2(G).$

On the algebra $\lambda(C_c(G)),$ the Plancherel weight $\tau$ is the functional
\[
    \tau(\lambda(f)) := f(1),\quad f \in C_c(G).
\]
where $1$ is the identity element of $G.$
It is not obvious that $\lambda$ is injective or that $\tau$ is well-defined. The fact that $\lambda$ is injective follows from \cite[Chapter VII, Theorem 3.4]{Takesaki_2}, and hence $\tau$ is well-defined on the algebra $\lambda(C_c(G)).$ See also \cite{Stinespring-tams-1959}.

The unimodularity of $G$ implies that $\tau$ is a trace, i.e. that
\[
    \tau(\lambda(f)\lambda(g)) = \tau(\lambda(g)\lambda(f)),\quad f,g \in C_c(G).
\]
Note also that
\begin{equation}\label{plancherel_identity}
    \tau(\lambda(f)^*\lambda(g)) = \int_G \overline{f(\gamma)}g(\gamma)\,d\nu(\gamma),\quad f,g \in C_c(G).
\end{equation}

The group von Neumann algebra $\VN(G)$ is defined to be closure of $\lambda(C_c(G))$
in the weak operator topology of $\mathcal{B}(L_2(G)).$ Equivalently, $\VN(G)$
is the weak operator topology closure of the linear span of the family $\{\lambda(\gamma)\}_{\gamma \in G}$ where
\[
    (\lambda(\gamma)u)(\eta) := u(\gamma^{-1}\eta),\quad u \in L_2(G),\; \gamma, \eta \in G.
\]
That is,
\[
    \VN(G) = \lambda(G)'' \subseteq \mathcal{B}(L_2(G))
\]
where $''$ denotes the double commutant \cite[Section 7.2.1]{Pedersen2018}. The commutant of $\VN(G)$ in $\Bc(L_2(G))$ is the closure of the right regular representation
of $G$, and it follows that a linear operator $T \in \Bc(L_2(G))$ belongs to $\VN(G)$ if and only if $T$
commutes with the action of $G$ on $L_2(G)$ by right multiplication.
For further details see e.g. \cite[Section 7.2.1]{Pedersen2018}, \cite[Chapter VII, Section 3]{Takesaki_2}, \cite[Section 18.8.1]{DixmierCStar1977}.

%

\begin{theorem}\cite[Theorem 7.2.7 \& Proposition 7.2.8]{Pedersen2018}
    The functional $\tau$ uniquely extends to a semifinite faithful normal trace on $\VN(G).$
\end{theorem}

By \eqref{plancherel_identity}, the mapping $f\mapsto \lambda(f)$ extends to an isometry
\[
    \lambda:L_2(G) \to L_2(\VN(G),\tau).
\]
This isometry is surjective \cite[Corollary 9.3]{Pedersen2018}. If $G$ is abelian, then $\VN(G)$ is isomorphic to the algebra of $L_{\infty}$-functions on the dual group $\widehat{G}$ and $\tau$ is integration with respect to the Haar measure on $\widehat{G}.$ In the abelian case, the isometric isomorphism $L_2(G)\cong L_2(\VN(G),\tau)$ is the Plancherel theorem. 

The simplest estimate for product-convolution operators in this setting is the following:
\begin{lemma}\label{L_2_cwikel}
For all $f \in L_2(G)$ and $T \in L_2(\VN(G),\tau),$ the composition
$
    M_f T
$
is meaningful in the sense that $T$ maps $L_2(G)$ into the domain of $M_f,$ and 
belongs to the Hilbert-Schmidt class $\mathcal{L}_2(L_2(G)).$ We have the norm identity
\[
    \|M_f T\|_{\mathcal{L}_2(L_2(G))} = \|f\|_{L_2(G)}\|T\|_{L_2(\VN(G),\tau)}.
\]
\end{lemma}
\begin{proof}
Note that if $T \in L_2(\VN(G),\tau),$ then $T = \lambda(g),$ where $g \in L_2(G)$.
By Young's convolution inequality \cite[Theorem 1.2.12]{Grafakos-1}, we have
\[
    \lambda(g):L_2(G)\to L_{\infty}(G)
\]
and therefore $M_f\lambda(g)$ makes sense as a bounded operator
\[
    M_f\lambda(g):L_2(G)\to L_2(G).
\]
Given that the Hilbert-Schmidt norm of $M_f\lambda(g)$ is the $L_2(G\times G)$ norm
of its kernel, we have
\[
    \|M_f \lambda(g)\|_{\mathcal{L}_2(L_2(G))}^2 = \int_{G\times G} |f(\gamma)g(\gamma \eta^{-1})|^2\,d\nu(\gamma)d\nu(\eta) = \int_G |f(\gamma)|^2\,d\nu(\gamma) \int_{G} |g(\gamma)|^2\,d\nu(\gamma).
\]
Thus, for all $f,g\in L_2(G).$
\[
    \|M_f \lambda(g)\|_{\mathcal{L}_2(L_2(G))} = \|f\|_{L_2(G)}\|g\|_{L_2(G)} = \|f\|_{L_2(G)}\|\lambda(g)\|_{L_2(\VN(G),\tau)}.
\]
\end{proof}

Lemma \ref{L_2_cwikel} can be combined with the obvious estimate
\[
    \|M_fT\|_{\infty} \leq \|f\|_{L_{\infty}(G)}\|T\|_{\VN(G)}
\]
to obtain Cwikel-type estimates for interpolation spaces between $\Lc_2(L_2(G))$ and $\Bc(L_2(G)).$

As already noted above in the proof of Lemma \ref{L_2_cwikel}, if $\xi\in L_2(G)$ we have
\[
    \|T\xi\|_{L_{\infty}(G)} \leq \|T\|_{L_2(\VN(G))}\|\xi\|_{L_2(G)}.
\]
and we have by definition,
\[
    \|T\xi\|_{L_2(G)} \leq \|T\|_{\infty}\|\xi\|_{L_2(G)}.
\]
By complex interpolation of noncommutative $L_p$ spaces, it follows that for all $p>2$ we have
\[
    \|T\xi\|_{L_q(G)} \leq \|T\|_{L_p(\VN(G))}\|\xi\|_{L_2(G)}
\]
where $q>2$ and
\[
    \frac12 = \frac1q+\frac1p.
\]
By the H\"older inequality, if $f \in L_p(G),$ we have
\[
    M_f\in \Bc(L_q(G), L_2(G))
\]
and therefore
\[
    M_fT\in \Bc(L_2(G),L_2(G)).
\]
Similarly, if $T \in L_{p,\infty}(\VN(G)),$ then
\[
    \|T\xi\|_{L_{q,\infty}(G)} \leq \|T\|_{L_{p,\infty}(\VN(G))}\|\xi\|_{L_2(G)}
\]
and therefore for $f \in L_p(G),$
\[
    M_fT \in \Bc(L_2(G),L_{2,\infty}(G)).
\]

The interpolation argument was carried out in an abstract setting in \cite{LeSZ-cwikel}.
The result is best stated in terms of the Calkin correspondence, which associates to a given rearrangement-invariant function space $E(0,\infty)$ on the semiaxis $(0,\infty)$
a unitarily invariant ideal $E(\Mv,\tau),$ where $(\Mv,\tau)$ is a semifinite von Neumann algebra. For further details on the Calkin correspondence in the semifinite setting, we direct the reader to \cite{LSZ}.

\begin{theorem}\label{general_cwikel_theorem}
Let $E(0,\infty)$ be a rearrangement invariant function space on $(0,\infty)$ which is an interpolation space for the couple $(L_2(0,\infty),L_{\infty}(0,\infty)).$ If $f\otimes T \in E(L_\infty(G)\bar{\otimes} \VN(G))$ then $M_f T \in E(\mathcal{B}(L_2(G)))$ and there exists a constant $C_E>0$ such that
\[
    \|M_f T\|_{E(\mathcal{B}(L_2(G)))} \leq C_E\|f\otimes T\|_{E(L_\infty(G)\bar{\otimes}\VN(G))}.
\]
In particular, let $2 < p < \infty,$ and let $f \in L_p(G).$ If $T\in L_p(\VN(G),\tau),$ then $M_f T\in \mathcal{L}_p(L_2(G))$ and there is a constant $C$ such that
\[
    \|M_f T\|_{\mathcal{L}_p(L_2(G))} \leq C\|f\|_{L_p(G)}\|T\|_{L_p(\VN(G),\tau)}.
\]
If instead $T\in L_{p,\infty}(\VN(G),\tau),$ then $M_fT \in \mathcal{L}_{p,\infty}(L_2(G))$ and there is a constant $C_p > 0$ potentially depending on $p$ such that
\[
    \|M_f T\|_{\mathcal{L}_{p,\infty}(L_2(G))} \leq C_p\|f\|_{L_p(G)}\|T\|_{L_{p,\infty}(\VN(G),\tau)}.
\]
%
\end{theorem}

\section{Functional calculus and specific Cwikel estimates}
In this section, we return to the setting of stratified Lie group with a fixed generating set $\{X_1,\ldots,X_m\}$ and sub-Laplacian $\Delta = \sum_{j=1}^m X_j^2.$

The operator $T$ in Theorem \ref{general_cwikel_theorem} is very general, and our primary interest is in operators of the form
\[
    M_fg(-\Delta)
\]
where $g$ is some bounded function on the positive semiaxis. 

\begin{lemma}
    If $g \in L_{\infty}(\Rl),$ then $g(-\Delta) \in \VN(G).$
\end{lemma}
\begin{proof}
    The statement of the lemma is that $-\Delta$ is affiliated with the von Neumann algebra $\VN(G).$ Since the commutant of $\VN(G)$
    is generated by the right regular representation of $G$ on $L_2(G),$ it suffices that $g(-\Delta)$ commutes with all right translations. Since $-\Delta$
    commutes with all right translations, this completes the proof.
\end{proof}

The fundamental fact about $g(-\Delta)$ we apply is a theorem of M.~Christ \cite[Proposition 3]{Christ1991}, which asserts that if $\check{g}$ is the convolution kernel of $g(-\Delta),$
then $\check{g} \in L_2(G)$ if and only if $g \in L_2(\Rl_+,t^{\frac{d_{\hom}-2}{2}}),$ and there exists a constant $c_G>0$ such that
\[
    \|\check{g}\|_{L_2(G)} = c_G\|g\|_{L_2(\Rl_+,t^{\frac{d_{\hom}-2}{2})}}.
\]
We place this into the context of noncommutative integration theory in the next lemma.
\begin{lemma}\label{christ_restatement}
    If $g$ is a Borel function on the semiaxis $\Rl_+,$ the operator $g(-\Delta)$ belongs to the noncommutative $L_2$ space $L_2(\VN(G),\tau)$ if and only
    if $g \in L_2(\Rl_+,t^{\frac{d_{\hom}-2}{2}}\,dt).$ Moreover, there exists a constant $c_G>0$ such that
    \[
        \|g(-\Delta)\|_{L_2(\VN(G),\tau)} = c_G\|g\|_{L_2(\Rl_+,t^{\frac{d_{\hom}-2}{2}}\,dt)}.
    \]
\end{lemma}
\begin{proof}
    This is simply a restatement of Christ's theorem recalled above. Indeed, according to \cite[Theorem 7.2.7]{Pedersen2018}, an element $x \in \VN(G)$ satisfies $\tau(|x|^2) < \infty$
    if and only if there exists $f\in L_2(G)$ such that $x = \lambda(f),$ and
    \[
        \tau(|x|^2) = \|f\|_{2}^2.
    \]  
    Taking $x = g(-\Delta),$ where $g \in L_{\infty}(\Rl_+),$ this result shows that $g(-\Delta) \in L_2(\VN(G),\tau)$ if and only if there exists $\check{g} \in L_2(G)$ such that $g(-\Delta)=\lambda(\check{g}),$ and
    in that case we have
    \[
        \tau(|g(-\Delta)|^2) = \|\check{g}\|_2^2.
    \]    
    Christ's result \cite[Proposition 3]{Christ1991} asserts
    that $g(-\Delta)$ has square integrable convolution kernel (i.e., $g(-\Delta) = \lambda(\check{g})$ for some $\check{g} \in L_2(G)$) if and only if
    \[
        \|\check{g}\|_{L_2(G)} = c_G\|g\|_{L_2(\Rl_+,t^{\frac{d_{\hom}-2}{2}}\,dt)} < \infty.
    \]
    Combining these results, it follows that for all $g \in L_{\infty}(\Rl_+),$ we have $g(-\Delta) \in L_2(\VN(G),\tau)$ if and only if $g \in L_2(\Rl_+,t^{\frac{d_{\hom}-2}{2}}\,dt)$
    and
    \[
        \|g(-\Delta)\|_{L_2(\VN(G),\tau)} = c_G\|g\|_{L_2(\Rl_+,t^{\frac{d_{\hom}-2}{2}}\,dt)}.
    \]
    Hence, the linear map $g\mapsto g(-\Delta)$ from $(L_2\cap L_{\infty})(\Rl_+,t^{\frac{d_{\hom}-2}{2}}\,dt)$ into $L_2(\VN(G),\tau)$ uniquely extends to an isometry from $L_2(\Rl_+,t^{\frac{d_{\hom}-2}{2}}\,dt),$
    and this completes the proof.
\end{proof}

\begin{corollary}\label{christ_trace_formula}
    For all $g \in L_1(\Rl_+,t^{\frac{d_{\hom}}{2}-1}\,dt),$ we have $g(-\Delta) \in L_1(\VN(G),\tau)$ and there exists a constant $c_G>0$ such that
    \[
        \tau(g(-\Delta)) = c_G\int_0^\infty g(t)t^{\frac{d_{\hom}}{2}-1}dt.
    \]
\end{corollary}
\begin{proof}
    Note that we have already proved in Lemma \ref{christ_restatement} that 
    \[
        \tau(|h(-\Delta)|^2) = \int_0^\infty |h(t)|^2 t^{\frac{d_{\hom}}{2}-1}dt,\quad h \in L_2(\Rl_+,t^{\frac{d_{\hom}}{2}-1}\,dt).
    \]
    Replacing the function $|h|^2$ with $g,$ we conclude that for all $0\leq g \in L_1(\Rl_+,t^{\frac{d_{\hom}}{2}-1}\,dt)$ we have $g(-\Delta)\in L_1(\VN(G),\tau)$ and
    \[
        \tau(g(-\Delta)) = \int_0^\infty g(t)t^{\frac{d_{\hom}}{2}-1}dt.
    \]
    Both sides of the above identity are linear functionals in $g.$ Given that every $g \in L_1(\Rl_+,t^{\frac{d_{\hom}}{2}-1}dt)$ is a linear combination
    of at most four positive functions, the result follows by linearity.    
\end{proof}

The next lemma shows that when $g$ has suitable integrability properties, $g(-\Delta)$ belongs to the noncommutative $L_p$-space $L_p(\VN(G),\tau).$
\begin{lemma}\label{g_Delta_L_p_bounds}
Let $0 < p \leq \infty.$ If $g \in L_p(\mathbb{R}_+,t^{\frac{d_{\hom}-2}{2}}\,dt)$ then $g(-\Delta) \in L_p(\VN(G),\tau)$ and there is a constant $c$ (implicitly depending on $\Delta$) such that
$$\|g(-\Delta)\|_{L_p(\VN(G),\tau)}=c_G^{\frac{2}{p}}\|g\|_{L_p(\mathbb{R}_+,t^{\frac{d_{\hom}-2}{2}}\,dt)}.$$
Similarly, if $g \in L_{p,\infty}(\mathbb{R}_+,t^{\frac{d_{\hom}-2}{2}}\,dt)),$ then $g(-\Delta) \in L_{p,\infty}(\VN(G),\tau)$ and
$$\|g(-\Delta)\|_{L_{p,\infty}(\VN(G),\tau)}=c_G^{\frac2p} \|g\|_{L_{p,\infty}(\mathbb{R}_+,t^{\frac{d_{\hom}-2}{2}}\,dt)}.$$
\end{lemma}
\begin{proof} Consider the linear mapping $g\mapsto g(-\Delta)$ from $(L_{\infty}(0,\infty),c_G^2t^{\frac{d_{\hom}}{2}-1}dt)$ to $({\rm VN}(G),\tau).$ Clearly, this mapping is a $\ast$-homomorphism. By Corollary \ref{christ_trace_formula}, this $\ast$-isomorphism preserves the trace, in particular it is injective. Since the mapping preserves the trace, it induces an isometry on all $L_p$ spaces and weak $L_p$ spaces.
\end{proof}

The preceding lemma combines with Lemma \ref{L_2_cwikel} and Theorem \ref{general_cwikel_theorem} (with $T = g(-\Delta)$) to give the following special case.
\begin{corollary}\label{L_p_specific_cwikel}
    Let $p \geq 2$ and $f \in L_p(G)$. If $g\in L_p(\mathbb{R}_+,t^{\frac{d_{\hom}-2}{2}}\,dt)$ then we have
    \[
        \|M_fg(-\Delta)\|_{\mathcal{L}_p(L_2(G))} \leq c_{p,G}\|f\|_{L_p(G)}\|g\|_{L_p(\mathbb{R}_+,t^{\frac{d_{\hom}-2}{2}}\,dt)}.
    \]
    For the $p=2$ case, we have
    \[
        \|M_fg(-\Delta)\|_{\mathcal{L}_2(L_2(G))} = c_G\|f\|_{L_2(G)}\|g\|_{L_2(\Rl_+,t^{\frac{d_{\hom}-2}{2}}\,dt)}.
    \]   
    If $p>2$ and $g \in L_{p,\infty}(\mathbb{R}_+,t^{\frac{d_{\hom}-2}{2}}\,dt),$ then
    \[
        \|M_fg(-\Delta)\|_{\mathcal{L}_{p,\infty}(L_2(G))} \leq c_{p,G}\|f\|_{L_p({G})}\|g\|_{L_{p,\infty}(\mathbb{R}_+,t^{\frac{d_{\hom}-2}{2}}\,dt)}.
    \]
\end{corollary}

Of particular interest will be the cases where $g(-\Delta)$ is a ``Riesz potential" $(-\Delta)^{-\frac{\beta}{2}}$ or a ``Bessel potential" $(1-\Delta)^{-\frac{\beta}{2}}.$ 

\begin{proof}[Proof of Theorem \ref{main_nontrivial_cwikel_theorem}.\eqref{mncta}] Let $p>2$ and set $g(t)=t^{-\frac{d_{\hom}}{2p}}.$ We claim that the function $g$ belongs to the space $ L_{p,\infty}(\mathbb{R}_+,t^{\frac{d_{\hom}}{2}-1}dt).$  Indeed, we have
$$\{s>0:\ g(s)>u\}=\{s>0:\ s^{-\frac{d_{\hom}}{2p}}>u\}=(0,u^{-\frac{2p}{d_{\hom}}}).$$
Therefore, for all $u>0$ we have
$$\int_{0}^\infty \chi_{\{s>0\;:\; g(s)>u\}}s^{\frac{d_{\hom}}{2}-1}\,ds=\int_0^{u^{-\frac{2p}{d_{\hom}}}}s^{\frac{d_{\hom}}{2}-1}ds=\frac{2}{d_{\hom}}u^{-p}.$$
Therefore
\[
    \|g\|_{L_{p,\infty}(\Rl_+,t^{\frac{d_{\hom}}{2}-1}\,dt)}^p = \sup_{u>0} u^p\int_{0}^\infty \chi_{\{s>0\;:\; g(s)>u\}}s^{\frac{d_{\hom}}{2}-1}\,ds = \frac{2}{d_{\hom}}.
\]
That is, $g \in L_{p,\infty}(\Rl_+,t^{\frac{d_{\hom}}{2}-1}\,dt)$ and thus Corollary \ref{L_p_specific_cwikel} implies the result.

\end{proof}

\section{Specific Cwikel-type theorems for $f\in C^{\infty}_c(G)$}

Theorem \ref{main_nontrivial_cwikel_theorem}.\eqref{mncta} with $\beta := \frac{d_{\hom}}{p}$ yields
$$f \in L_{\frac{d_{\hom}}{\beta}}(G)\Longrightarrow M_f(1-\Delta)^{-\frac{\beta}{2}} \in \Lc_{\frac{d_{\hom}}{\beta},\infty}(L_2(G))$$
provided that $\beta < \frac{d_{\hom}}{2}$ (i.e., $p > 2$).
In general this implication is false when $\beta\geq \frac{d_{\hom}}{2}.$ Even in the case where $G = \Rl,$ there are known counterexamples \cite[Theorem 5.1]{LeSZ-cwikel}.

Nevertheless it is possible to achieve results for $\beta \geq \frac{d_{\hom}}{2}$ by placing further restrictions on $f.$ We will begin with highly restrictive assumptions on $f,$ which will be weakened in the next section.

For brevity, denote by $J$ the operator
$$J = (1-\Delta)^{\frac12}.$$

In this section we prove the following assertion:
\begin{theorem}\label{MSX_commutator_estimate} Let $\alpha,\beta,\gamma\in \Rl$ and let $f \in C^\infty_c(G).$ 
\begin{enumerate}[{\rm (i)}]
\item\label{msxa} if $\alpha+\gamma=0,$ then
$$J^{-\alpha}M_fJ^{-\gamma}\in\Bc(L_2(G)).$$
\item\label{msxb} if $\alpha+\gamma+1=\beta,$ then
$$J^{-\alpha}[J^{\beta},M_f]J^{-\gamma} \in \Bc(L_2(G)).$$
\item\label{msxc} if $\alpha+\gamma>0,$ then
$$J^{-\alpha}M_fJ^{-\gamma}\in\Lc_{\frac{d_{\hom}}{\alpha+\gamma},\infty}.$$
\item\label{msxd} if $\alpha+\gamma+1>\beta,$ then
$$ J^{-\alpha}[J^{\beta},M_f]J^{-\gamma} \in\Lc_{\frac{d_{\hom}}{\alpha+\gamma+1-\beta},\infty}(L_2(G)).$$
\end{enumerate}	
\end{theorem}

Theorem \ref{MSX_commutator_estimate} is proved in several steps. We begin with the following Lemma, which can be viewed as a quantitative form of a special case of part \eqref{msxc} of the Theorem. We will state the result in terms
of Sobolev spaces. Following the definition in \cite{Folland1975}, for $s \in \Rl$ and $1\leq p \leq \infty$ we define
\[
    \|f\|_{W^s_p(G)} := \|(1-\Delta)^{\frac{s}{2}}f\|_{L_p(G)}.
\]
It was proved by Folland in \cite[Corollary 4.13]{Folland1975} that for $s \in \mathbb{N}$ and $1 < p < \infty$ this norm is equivalent to
\[
    \|f\|_{L_p(G)}+\sum_{k=1}^s\sum_{1\leq j_1,\ldots,j_k\leq m} \|X_{j_1}X_{j_2}\cdots X_{j_k}f\|_{L_p(G)}.
\]

\begin{lemma}\label{msx integer estimate} Let $n\in\mathbb{Z}_+,$ $\alpha\in[0,\frac{d_{\hom}}{2})$ and let $f\in C^{\infty}_c(G).$ We have
$$\|J^{2n}M_fJ^{-2n-\alpha}\|_{\frac{d_{\hom}}{\alpha},\infty}\leq c_{n,d_{\hom},\alpha}\|f\|_{W^{2n}_{\frac{d_{\hom}}{\alpha}}(G)}.$$
\end{lemma}
\begin{proof} We prove the assertion by induction on $n.$ The base of induction (i.e. $n=0$) follows directly from Theorem \ref{main_nontrivial_cwikel_theorem}.\eqref{mncta}. We concentrate on the step of induction. Denote for brevity
$$|||f|||_{(n,\alpha)}=\|J^{2n}M_fJ^{-2n-\alpha}\|_{\frac{d_{\hom}}{\alpha},\infty}.$$
	
We write
$$J^{2n+2}M_fJ^{-2n-2-\alpha}=J^{2n}M_fJ^{-2n-\alpha}+J^{2n}[J^2,M_f]J^{-2n-2-\alpha}.$$
Clearly,
$$[J^2,M_f]=-\sum_{l=1}^m[X_l^2,M_f]=-\sum_{l=1}^m(X_lM_{X_lf}+M_{X_lf}X_l)=-2\sum_{l=1}^mM_{X_lf}X_l-M_{\Delta f}.$$
Thus,
\begin{align*}
J^{2n+2}M_fJ^{-2n-2}&=J^{2n}M_fJ^{-2n-\alpha}-2\sum_{l=1}^mJ^{2n}M_{X_lf}J^{-2n-\alpha}\cdot J^{2n+\alpha}X_lJ^{-2n-2-\alpha}\\
                    &\quad -J^{2n}M_{\Delta f}J^{-2n-\alpha}\cdot J^{-2}.
\end{align*}
By quasi-triangle inequality and H\"older inequality, we have
\begin{align*}
\|J^{2n+2}M_fJ^{-2n-2}\|_{\frac{d_{\hom}}{\alpha},\infty}&\leq c'_{n,d_{\hom},\alpha}\cdot\Big(\|J^{2n}M_fJ^{-2n-\alpha}\|_{\frac{d_{\hom}}{\alpha},\infty}\\
&\quad +2\sum_{l=1}^m\|J^{2n}M_{X_lf}J^{-2n-\alpha}\|_{\frac{d_{\hom}}{\alpha},\infty}\cdot\|J^{2n+\alpha}X_lJ^{-2n-2-\alpha}\|_{\infty}\\
&\quad +\|J^{2n}M_{\Delta f}J^{-2n-\alpha}\|_{\frac{d_{\hom}}{\alpha},\infty}\Big).
\end{align*}
Equivalently,
$$\|f\|_{(n+1,\alpha)}\leq c''_{n,d_{\hom},\alpha}\cdot\Big(\|f\|_{(n,\alpha)}+\sum_{l=1}^m\|X_lf\|_{(n,\alpha)}+\|\Delta f\|_{(n,\alpha)}\Big).$$
By inductive assumption, we have
\begin{align*}
\|f\|_{(n,\alpha)}&\leq c_{n,d_{\hom},\alpha}\|f\|_{W^{2n}_{\frac{d_{\hom}}{\alpha}}},\\
\|X_lf\|_{(n,\alpha)}&\leq c_{n,d_{\hom},\alpha}\|X_lf\|_{W^{2n}_{\frac{d_{\hom}}{\alpha}}},\quad 1\leq l\leq m,\\
\|\Delta f\|_{(n,\alpha)}&\leq c_{n,d_{\hom},\alpha}\|f\|_{W^{2n}_{\frac{d_{\hom}}{\alpha}}}.
\end{align*}
Thus,
$$\|f\|_{(n+1,\alpha)}\leq c''_{n,d_{\hom},\alpha}c_{n,d_{\hom},\alpha}\cdot\Big(\|f\|_{W^{2n}_{\frac{d_{\hom}}{\alpha}}}+\sum_{l=1}^m\|X_lf\|_{W^{2n}_{\frac{d_{\hom}}{\alpha}}}+\|\Delta f\|_{W^{2n}_{\frac{d_{\hom}}{\alpha}}}\Big).$$
This establishes the step of induction and, hence, the assertion.
\end{proof}

\begin{proof}[Proof of Theorem \ref{MSX_commutator_estimate}.\eqref{msxa}] Without loss of generality, $\alpha>0.$ Fix $n\in\mathbb{Z}_+$ such that $\alpha\in[2n,2n+2].$ Define an analytic operator-valued function $F$ on the strip $\{2n\leq\Re(z)\leq 2n+2\}$ by setting
$$F(z)=J^zM_fJ^{-z},\quad \Re(z)\in [2n,2n+2].$$
By Lemma \ref{msx integer estimate}, we have
$$\sup_{\Re(z)=2n}\|F(z)\|_{\infty}=\|F(2n)\|_{\infty}\leq c_{n,d_{\hom}}\|f\|_{W^{2n}_{\infty}}\leq c_{n,d_{\hom}}\|f\|_{W^{2n+2}_{\infty}}$$
and
$$\sup_{\Re(z)=2n+2}\|F(z)\|_{\infty}=\|F(2n+2)\|_{\infty}\leq c_{n+1,d_{\hom}}\|f\|_{W^{2n+2}_{\infty}}.$$
Observe also that the function $\|F(z)\|_{\infty}$ is trivially uniformly bounded for $\Re(z)\in [2n,2n+2].$ This permits the application of the Phragm\`{e}n-Lindel\"{o}f maximum modulus principle, which gives the upper bound
$$\|F(\alpha)\|_{\infty}\leq \max\{c_{n,d_{\hom}},c_{n+1,d_{\hom}}\}\|f\|_{W^{2n+2}_{\infty}}.$$
\end{proof}

\begin{proof}[Proof of Theorem \ref{MSX_commutator_estimate}.\eqref{msxc}] Suppose first that $\alpha+\gamma<\frac{d_{\hom}}{2}.$ Without loss of generality, $\alpha>0.$ Fix $n\in\mathbb{Z}_+$ such that $\alpha\in[2n,2n+2].$ Set $\theta=\alpha+\gamma.$ Define analytic operator-valued function $F$ on the strip $\{2n\leq\Re(z)\leq 2n+2\}$ by setting
$$F(z)=J^zM_fJ^{-z-\theta},\quad z\in [2n,2n+2].$$
By Lemma \ref{msx integer estimate}, we have
$$\sup_{\Re(z)=2n}\|F(z)\|_{\frac{d_{\hom}}{\theta},\infty}=\|F(2n)\|_{\frac{d_{\hom}}{\theta},\infty}\leq c_{n,d_{\hom},\theta}\|f\|_{W^{2n}_{\frac{d_{\hom}}{\theta}}}\leq c_{n,d_{\hom},\theta}\|f\|_{W^{2n+2}_{\frac{d_{\hom}}{\theta}}}$$
and
$$\sup_{\Re(z)=2n+2}\|F(z)\|_{\frac{d_{\hom}}{\theta},\infty}=\|F(2n+2)\|_{\frac{d_{\hom}}{\theta},\infty}\leq c_{n+1,d_{\hom},\theta}\|f\|_{W^{2n+2}_{\frac{d_{\hom}}{\theta}}}.$$
As in the proof of Theorem \ref{MSX_commutator_estimate}.\eqref{msxa}, is is trivial that $\|F(z)\|_{\frac{d_{\hom}}{\theta},\infty}$ is uniformly bounded for $\Re(z) \in [2n,2n+2].$
This again pemits the application of the Phragm\`{e}n-Lindel\"{o}f maximum modulus principle, which yields
$$\|F(\alpha)\|_{\infty}\leq \max\{c_{n,d_{\hom},\theta},c_{n+1,d_{\hom},\theta}\}\|f\|_{W^{2n+2}_{\frac{d_{\hom}}{\theta}}}.$$
This proves the assertion when $0 < \alpha+\gamma<\frac{d_{\hom}}{2}.$ We now remove the upper bound on $\alpha+\gamma.$
%

Now define a sequence $\{\theta_k\}_{k=1}^r$ as follows. Take
\[
    \theta_k = \frac{k d_{\hom}}{4}-\alpha,\quad 1\leq k\leq r
\]
where $r$ is chosen as the least integer such that that $0 < \gamma +\alpha-\frac{rd_{\hom}}{4} < \frac{d_{\hom}}{2}.$
By construction,
\[
    0 < \alpha+\theta_1 < \frac{d_{\hom}}{2},\quad 0<\gamma-\theta_r<\frac{d_{\hom}}{2}
\]
and
\[
    0 < -\theta_k+\theta_{k+1} < \frac{d_{\hom}}{2},\quad 1\leq k<r.
\]
Write $f=\prod_{k=1}^rf_k,$ where $f_k\in C^{\infty}_c(G).$ We factorise $J^{-\alpha}M_fJ^{-\gamma}$ as follows:
\[
    J^{-\alpha}M_fJ^{-\gamma}= J^{-\alpha} M_{f_1} J^{-\theta_1}\cdot J^{\theta_1}M_{f_2}J^{-\theta_2} \cdot J^{\theta_2}M_{f_2}J^{-\theta_3} \cdots J^{\theta_r}M_{f_r}J^{-\gamma}.
\]
We have
\[
    J^{-\alpha}M_{f_1}J^{-\theta_1} \in \Lc_{\frac{d_{\hom}}{\theta_1+\alpha},\infty},\quad J^{\theta_r}M_{f_r}J^{-\gamma} \in \Lc_{\frac{d_{\hom}}{\gamma-\theta_r},\infty}.
\]
and for all $1\leq k\leq r-1$ we have
\[
    J^{\theta_k}M_{f_k}J^{-\theta_{k+1}} \in \Lc_{\frac{d_{\hom}}{\theta_{k+1}-\theta_k},\infty}.
\]
By H\"older's inequality, we get
\[
    J^{-\alpha}M_fJ^{-\gamma} \in \Lc_{\frac{d_{\hom}}{\alpha+\theta_1},\infty}\cdot \Lc_{\frac{d_{\hom}}{\theta_2-\theta_1},\infty}\cdots \Lc_{\frac{d_{\hom}}{\gamma-\theta_{r}},\infty} \subseteq \Lc_{\frac{d_{\hom}}{\alpha+\gamma},\infty}.
\]
\end{proof}

In the next lemma, we use the notion of double operator integrals as specifically used in \cite{DDSZ-2020-II} and \cite{Hiai-Kosaki-book-2003}. For extra details concerning
double operator integration theory, see \cite{PS-crelle,SkripkaTomskova}. In the following lemmas, $H$ is an arbitrary separable Hilbert space.
\begin{lemma}\label{phibeta lemma} Let $0<\beta<2$ and let
$$\phi_{\beta}(\lambda,\mu)=\frac{\lambda^{\beta}-\mu^{\beta}}{\lambda^2-\mu^2}\lambda^{1-\frac{\beta}{2}}\mu^{1-\frac{\beta}{2}},\quad \lambda,\mu>0.$$
Let $B$ be a potentially unbounded positive self-adjoint operator on $H$ with trivial kernel. We have
$$\|T^{B,B}_{\phi_{\beta}}\|_{\Bc(H)\to\Bc(H)}\leq c_{\beta}.$$
\end{lemma}
\begin{proof} By assumption,
$$\phi_{\beta}(\lambda,\mu)=f(\frac{\lambda}{\mu}),\quad\lambda,\mu>0,$$
where
$$f(e^t)=\frac{e^{\beta t}-1}{e^{2t}-1}e^{(1-\frac{\beta}{2})t} = \frac{\sinh(\frac{\beta}{2}t)}{\sinh(t)},\quad t\in\mathbb{R}.$$ 
Thus, $f\circ\exp$ is a Schwartz class function on $\Rl,$ and hence has integrable Fourier transform. The boundedness on $\Bc(H)$ of double operator integrals with symbol $f(\frac{\lambda}{\mu})$ where $f$ has integrable Fourier transform
is proved in e.g. \cite[Lemma 9]{PS-crelle}.

\end{proof}

The proof of the next lemma uses some of the functional calculus properties of the double operator integral transformer $T^{B,B}_{\phi_{\beta}},$ for this we refer to \cite{DDSZ-2020-II}.
\begin{lemma}\label{doi bounded above and below}
Let $A,B\in\Bc(H)$ and let there exists a constant $c>0$ such that $c\leq B\leq c^{-1}.$ We have
$$B^{-\alpha}[B^{\beta},A]B^{-\gamma}=T^{B,B}_{\phi_{\beta}}\Big(B^{\frac{\beta}{2}-\alpha-1}[B^2,A]B^{\frac{\beta}{2}-\gamma-1}\Big),\quad 0<\beta<2.$$
\end{lemma}
\begin{proof} Note that
$$B^{-\alpha}[B^{\beta},A]B^{-\gamma}=T^{B,B}_{\frac{\lambda^{\beta}-\mu^{\beta}}{\lambda^{\alpha}\mu^{\gamma}}}(A),\quad B^{\frac{\beta}{2}-\alpha-1}[B^2,A]B^{\frac{\beta}{2}-\gamma-1}=T^{B,B}_{\lambda^{\frac{\beta}{2}-\alpha-1}(\lambda^2-\mu^2)\mu^{\frac{\gamma}{2}-\alpha-1}}(A).$$
Clearly,
$$\frac{\lambda^{\beta}-\mu^{\beta}}{\lambda^{\alpha}\mu^{\gamma}}=\phi_{\beta}(\lambda,\mu)\cdot\lambda^{\frac{\beta}{2}-\alpha-1}(\lambda^2-\mu^2)\mu^{\frac{\gamma}{2}-\alpha-1}.$$
By Lemma \ref{phibeta lemma}, the double operator integral with symbol $\phi_{\beta}$ is bounded on $\Bc(H).$ Thus,
\begin{align*}
    B^{-\alpha}[B^{\beta},A]B^{-\gamma}&=T^{B,B}_{\frac{\lambda^{\beta}-\mu^{\beta}}{\lambda^{\alpha}\mu^{\gamma}}}(A)\\
                                       &=T^{B,B}_{\phi_{\beta}(\lambda,\mu)\cdot \lambda^{\frac{\beta}{2}-\alpha-1}(\lambda^2-\mu^2)\mu^{\frac{\gamma}{2}-\alpha-1}}(A)\\
                                       &=T^{B,B}_{\phi_{\beta}(\lambda,\mu)}\Big(T^{B,B}_{\lambda^{\frac{\beta}{2}-\alpha-1}(\lambda^2-\mu^2)\mu^{\frac{\gamma}{2}-\alpha-1}}(A)\Big)=T^{B,B}_{\phi_{\beta}(\lambda,\mu)}\Big(B^{\frac{\beta}{2}-\alpha-1}[B^2,A]B^{\frac{\beta}{2}-\gamma-1}\Big).
\end{align*}
\end{proof}

\begin{lemma}\label{jsquare commutator lemma}
Let $B\geq0$ be a (potentially unbounded) self-adjoint linear operator on a Hilbert space $H$ with $\ker(B)=0$, and let $A$ be a bounded operator on $H.$ Let $\beta\in(0,2)$ and let $\alpha,\gamma\in\mathbb{R}.$ Assume that $A(\dom(B^2))\subset \dom(B^2),$ and that
$$B^{\frac{\beta}{2}-\alpha-1}[B^2,A]B^{\frac{\beta}{2}-\gamma-1}\in\Bc(H),$$
then
$$B^{-\alpha}[B^{\beta},A]B^{-\gamma}\in\Bc(H)$$
and
$$\Big\|B^{-\alpha}[B^{\beta},A]B^{-\gamma}\Big\|_{\infty}\leq c_{\beta}\Big\|B^{\frac{\beta}{2}-\alpha-1}[B^2,A]B^{\frac{\beta}{2}-\gamma-1}\Big\|_{\infty}.$$
\end{lemma}
\begin{proof} Let $p_n=\chi_{(\frac1n,n)}(B).$ Consider Hilbert space $p_n(H)$ and the operators $A_n=p_nAp_n$ and $B_n=Bp_n$ on $p_n(H).$ Clearly, $\frac1n \leq B_n\leq n$ on $p_n(H).$ Note that
$$p_n\cdot B^{-\alpha}[B^{\beta},A]B^{-\gamma}\cdot p_n=B_n^{-\alpha}[B_n^{\beta},A_n]B_n^{-\gamma},$$
$$p_n\cdot B^{\frac{\beta}{2}-\alpha-1}[B^2,A]B^{\frac{\beta}{2}-\gamma-1}\cdot p_n=B_n^{\frac{\beta}{2}-\alpha-1}[B_n^2,A_n]B_n^{\frac{\beta}{2}-\gamma-1}.$$
By Lemma \ref{doi bounded above and below}, we have
$$B_n^{-\alpha}[B_n^{\beta},A_n]B_n^{-\gamma}=T^{B_n,B_n}_{\phi_{\beta}}\Big(B_n^{\frac{\beta}{2}-\alpha-1}[B_n^2,A_n]B_n^{\frac{\beta}{2}-\gamma-1}\Big).$$
Thus,
$$p_n\cdot B^{-\alpha}[B^{\beta},A]B^{-\gamma}\cdot p_n=T^{B_n,B_n}_{\phi_{\beta}}\Big(p_n\cdot B^{\frac{\beta}{2}-\alpha-1}[B^2,A]B^{\frac{\beta}{2}-\gamma-1}\cdot p_n\Big).$$
By Lemma \ref{phibeta lemma}, we have
$$\Big\|p_n\cdot B^{-\alpha}[B^{\beta},A]B^{-\gamma}\cdot p_n\Big\|_{\infty}\leq c_{\beta}\Big\|p_n\cdot B^{\frac{\beta}{2}-\alpha-1}[B^2,A]B^{\frac{\beta}{2}-\gamma-1}\cdot p_n\Big\|_{\infty}.$$
Hence,
$$\Big\|p_n\cdot B^{-\alpha}[B^{\beta},A]B^{-\gamma}\cdot p_n\Big\|_{\infty}\leq c_{\beta}\Big\|B^{\frac{\beta}{2}-\alpha-1}[B^2,A]B^{\frac{\beta}{2}-\gamma-1}\Big\|_{\infty}.$$
Passing $n\to\infty,$ we complete the proof.
\end{proof}

\begin{lemma}\label{beta2 lemma} Theorem \ref{MSX_commutator_estimate}.\eqref{msxb} holds for $\beta=2.$ That is, for all $\alpha,\gamma\in \mathbb{R}$ such that $\alpha+\gamma=1$ and $f \in C^\infty_c(G)$ we have
\[
    J^{-\alpha}[J^2,M_f]J^{-\gamma} \in \Bc(L_2(G)).
\]
\end{lemma}
\begin{proof} 
First, observe that by the Leibniz rule, the multiplier $M_f$ maps $\dom(J^2) = H^2(G)$
into itself and the commutator $[J^2,M_f]$ is meaningful on $\dom(J^2).$

Indeed, 
$$[J^2,M_f]=2\sum_{l=1}^mX_lM_{X_lf}+M_{\Delta f}.$$
Taking into account that $\alpha+\gamma=1,$ we infer that
$$J^{-\alpha}[J^2,M_f]J^{-\gamma}=2\sum_{l=1}^mJ^{-\alpha}X_lJ^{\alpha-1}\cdot J^{\gamma}M_{X_lf}J^{-\gamma}+J^{\gamma-1}M_{\Delta f}J^{-\gamma}.$$
By the triangle and H\"older inequalities, we have
$$\Big\|J^{-\alpha}[J^2,M_f]J^{-\gamma}\Big\|_{\infty}\leq 2\sum_{l=1}^m\Big\|J^{-\alpha}X_lJ^{\alpha-1}\|_{\infty}\|J^{\gamma}M_{X_lf}J^{-\gamma}\|_{\infty}+\|J^{\gamma}M_{\Delta f}J^{-\gamma}\|_{\infty}.$$
The assertion follows now from Theorem \ref{MSX_commutator_estimate}.\eqref{msxa}.
\end{proof}

\begin{proof}[Proof of Theorem \ref{MSX_commutator_estimate}.\eqref{msxb}] Suppose $0<\beta<2.$ By Lemma \ref{jsquare commutator lemma}, we have
\begin{equation}\label{jsquare_applied_to_msxb}
    \Big\|J^{-\alpha}[J^{\beta},M_f]J^{-\gamma}\Big\|_{\infty}\leq c_{\beta}\Big\|J^{\frac{\beta}{2}-\alpha-1}[J^2,M_f]J^{\frac{\beta}{2}-\gamma-1}\Big\|_{\infty}.
\end{equation}
Clearly,
$$(\alpha+1-\frac{\beta}{2})+(\gamma+1-\frac{\beta}{2})=\alpha+\gamma-\beta+2=2-1.$$
Hence, Lemma \ref{beta2 lemma} with the parameter $\alpha$ set as $\alpha+1-\frac{\beta}{2}$ and $\beta$ set as $\gamma+1-\frac{\beta}{2}$ yields
\[
    J^{\frac{\beta}{2}-\alpha-1}[J^2,M_f]J^{\frac{\beta}{2}-\gamma-1} \in \Bc(L_2(G)).
\]
Thus \eqref{jsquare_applied_to_msxb} implies that $J^{-\alpha}[J^{\beta},M_f]J^{-\gamma} \in \Bc(L_2(G))$ and this completes the proof in the case $0 < \beta< 2.$

Suppose now $\beta>0.$ Choose $n\in\mathbb{N}$ such that $\beta<n.$ By Leibniz rule, we write	
$$[J^{\beta},M_f]=\sum_{k=0}^{n-1}J^{\frac{\beta k}{n}}[J^{\frac{\beta}{n}},M_f]J^{\frac{\beta(n-1-k)}{n}}$$
and
$$J^{-\alpha}[J^{\beta},M_f]J^{-\gamma} =\sum_{k=0}^{n-1}J^{\frac{\beta k}{n}-\alpha}[J^{\frac{\beta}{n}},M_f]J^{\frac{\beta(n-1-k)}{n}-\gamma}.$$	
Clearly,
$$(\alpha-\frac{\beta k}{n})+(\gamma-\frac{\beta(n-1-k)}{n})=\alpha+\gamma-\frac{\beta(n-1)}{n}=\frac{\beta}{n}-1.$$
By construction of $n,$ we have $0 < \frac{\beta}{n} < 1,$ and hence for every $0\leq k\leq n-1$ the case just proved implies that
\[
    J^{\frac{\beta k}{n}-\alpha}[J^{\frac{\beta}{n}},M_f]J^{\frac{\beta(n-1-k)}{n}-\gamma} \in \Bc(L_2(G)).
\]
This concludes the argument for $\beta>0.$
	
Suppose now $\beta<0.$ We have
$$J^{-\alpha}[J^{\beta},M_f]J^{-\gamma}=-J^{\beta-\alpha}[J^{-\beta},M_f]J^{\beta-\gamma}.$$
Clearly,
$$(\alpha-\beta)+(\gamma-\beta)=\alpha+\gamma-2\beta=-\beta-1.$$	
This reduces the assertion to the $\beta>0$ case. This concludes the argument for $\beta<0.$
\end{proof}

\begin{proof}[Proof of Theorem \ref{MSX_commutator_estimate}.\eqref{msxd}] We write
$$J^{-\alpha}[J^{\beta},M_f]J^{-\gamma}=J^{\beta-\alpha}M_fJ^{-\gamma}-J^{-\alpha}M_fJ^{\beta-\gamma}.$$
Next, let $\psi\in C^{\infty}_c(G)$ be such that $f=f\psi.$ We have
$$J^{\beta-\alpha}M_fJ^{-\gamma}-M_{\psi}\cdot J^{\beta-\alpha}M_fJ^{-\gamma}=[J^{\beta-\alpha},M_{\psi}]J^{\alpha-\beta+1}\cdot J^{\beta-\alpha-1}M_fJ^{-\gamma},$$
$$J^{-\alpha}M_fJ^{\beta-\gamma}-M_{\psi}\cdot J^{-\alpha}M_fJ^{\beta-\gamma}=[J^{-\alpha},M_{\psi}]J^{\alpha+1}\cdot J^{-\alpha-1}M_fJ^{\beta-\gamma}.$$
The first factors on the right-hand-side are bounded by Theorem \ref{MSX_commutator_estimate}.\eqref{msxb}. The second factors belong to $\mathcal{L}_{\frac{d_{\hom}}{\alpha+\gamma-\beta+1},\infty}$ by Theorem \ref{MSX_commutator_estimate}.\eqref{msxc}. Therefore,
$$J^{-\alpha}[J^{\beta},M_f]J^{-\gamma}-M_{\psi}\cdot J^{-\alpha}[J^{\beta},M_f]J^{-\gamma}\in \mathcal{L}_{\frac{d_{\hom}}{\alpha+\gamma-\beta+1},\infty}.$$
	
Now, set $\theta=\alpha+\gamma+1-\beta$ and $\alpha'=\alpha-\theta.$ We write
$$M_{\psi}\cdot J^{-\alpha}[J^{\beta},M_f]J^{-\gamma}=M_{\psi}J^{-\theta}\cdot J^{-\alpha'}[J^{\beta},M_f]J^{-\gamma}.$$
The first factor on the right-hand-side belongs to $\Lc_{\frac{d_{\hom}}{\theta},\infty}$ by Theorem \ref{MSX_commutator_estimate}.\eqref{msxc}. The second factor is bounded by Theorem \ref{MSX_commutator_estimate}.\eqref{msxb}. The assertion follows now from H\"older inequality.
\end{proof}

\section{Specific Cwikel-type estimates}
A consequence of Theorem \ref{MSX_commutator_estimate}.\eqref{msxc} is that for all $p>0$ and $f \in C^\infty_c(G),$ we have
\[
    M_f(1-\Delta)^{-\frac{d_{\hom}}{2p}} \in \Lc_{p,\infty}.
\]
When $p > 2,$ this assertion can be made quantitative thanks to Corollary \ref{L_p_specific_cwikel}, which supplies the estimate
\[
    \|M_f(1-\Delta)^{-\frac{d_{\hom}}{2p}}\|_{p,\infty} \leq c_{p,G}\|f\|_{L_p(G)}.
\]
We seek similar quantitative results for $p \leq 2.$ Our first result in this direction assumes that $f$ is compactly supported. 

\begin{lemma}\label{less_simple_cwikel}
Let $p\leq 2,$ $q>2$ and $\frac1p=\frac1q+\frac1r.$ If $f \in L_q(G)$ and $f$ is supported in some compact subset $K$ of $G,$ then $M_f(1-\Delta)^{-\frac{d_{\hom}}{2p}}\in \Lc_{p,\infty}$ and
$$\|M_f(1-\Delta)^{-\frac{d_{\hom}}{2p}}\|_{p,\infty} \leq C_{K,p,q}\|f\|_{L_q(G)}.$$
\end{lemma}
\begin{proof}
Let $\psi\in C^\infty_c(G)$ be equal to $1$ on $K,$ so that $M_f = M_fM_{\psi}.$ We now write 
$$M_f(1-\Delta)^{-\frac{d_{\hom}}{2p}}= M_f(1-\Delta)^{-\frac{d_{\hom}}{2q}}\cdot (1-\Delta)^{\frac{d_{\hom}}{2q}}M_{\psi}(1-\Delta)^{-\frac{d_{\hom}}{2p}}.$$
Theorem \ref{MSX_commutator_estimate}.\eqref{msxc} asserts that
$$(1-\Delta)^{\frac{d_{\hom}}{2q}}M_{\psi}(1-\Delta)^{-\frac{d_{\hom}}{2p}} \in \Lc_{r,\infty}.$$
Theorem \ref{main_nontrivial_cwikel_theorem}.\eqref{mncta} and the assumption $q >2$ yield
$$\|M_f(1-\Delta)^{-\frac{d_{\hom}}{2q}}\|_{q,\infty}\leq C_{G,q}\|f\|_{q}.$$
The desired result follows from  from H\"older's inequality and the preceding assertions.
\end{proof}

The classical estimate on $\Rl^d,$ for $d\geq 1$ corresponding to Lemma \ref{less_simple_cwikel} is due to Birman and Solomyak, which implies that if $p<2$ 
and $f$ is supported in a compact $K$ subset of $\Rl^d,$ then
\[
    \|M_f(1-\Delta)^{-\frac{d}{2p}}\|_{p,\infty} \leq C_{K,p}\|f\|_{L_2(K)}.
\]
See  \cite[Lemma 4.2]{LeSZ-cwikel}. Here, $\Delta$ is the usual Laplace operator on $\Rl^d.$
Lemma \ref{less_simple_cwikel} is weaker than this in the sense that we take $f \in L_q(K)$ for some $q>2$ rather than simply $f \in L_2(K).$

The next step we take in weakening the assumptions on $f$ is to remove the requirement that $f$ be compactly supported. This is best understood in comparison
to the classical case. Classically, one proves that if $f$ is supported in the $d$-cube $[0,1]^d$ then
\[
    \|M_f(1-\Delta)^{-\frac{d}{2p}}\|_{p,\infty} \leq C_p\|f\|_{L_2([0,1]^d)}.
\]
Since the Laplace operator $\Delta$ is translation invariant, if $f$ is supported in the set $x+[0,1]^d$ for some $x \in \Rl^d$ we have
\[
    \|M_f(1-\Delta)^{-\frac{d}{2p}}\|_{p,\infty} \leq C_p\|f\|_{L_2(x+[0,1]^d)}.
\]
When $p<2,$ it was originally proved by Birman and Solomyak that
\[
    \|M_f(1-\Delta)^{-\frac{d}{2p}}\|_{p,\infty} \leq C_p\left(\sum_{n\in \Itgr^d} \|f\|_{L_2(n+[0,1]^d)}^p\right)^{\frac1p} =: C_{\beta}\|f\|_{\ell_p(L_2)(\Rl^d)}.
\]
See \cite{LeSZ-cwikel} and references therein.
We wish to obtain a similar estimate in the setting of stratified Lie groups. The main obstacle is to find an analogy for the decomposition
\[
    \Rl^d = \bigcup_{n\in \Itgr^d} n+[0,1)^d.
\]
A direct analogy would be to select a countable subgroup $\Gamma$ of $G$ such that the quotient space $G/\Gamma$ is compact, and consider the decomposition
\[
    G = \bigcup_{\gamma\in \Gamma} \gamma F
\]
where $F$ is a fundamental domain for $\Gamma.$ However, this is not always possible. In fact a simply connected nilpotent Lie group $G$ admits a discrete subgroup
$\Gamma$ such that $G/\Gamma$ is compact if and only if $\mathfrak{g}$ admits a basis for which the structure constants are rational numbers \cite[Theorem 2.12]{Raghunathan1972}.

Instead, we seek a compact subset $K$ of $G$ and a countable subset $\Gamma$ such that
\[
    G = \bigcup_{\gamma\in \Gamma} \gamma K
\]
and such that the sets $\{\gamma K\}_{\gamma\in \Gamma}$ do not overlap too much. Note that in this decomposition, $\Gamma$ is not necessarily a subgroup. To this end we use a new decomposition which we call a \emph{covering of bounded multiplicity}, defined in Lemma \ref{locally_finite_coverings_exist} below.

\subsection{Coverings of bounded multiplicity}

In the following lemmas we use a translation-invariant quasi-metric on $G$, which may be described as follows.
For $g \in G,$ we have $g = \sum_{k=1}^n p_k(g),$ where $p_k(g)$ is the component of $g$ in $\gf_k,$ and $n$ is the number of steps in the stratification of $\gf.$ Denote
$$\rho(g) := \max_{1\leq k\leq n}\|p_k(g)\|_{\infty}^{\frac1k}.$$
Here, $\|\cdot\|_\infty$ refers to the norm as a linear operator on $\gf.$
Note that $\rho(g)=\rho(g^{-1})$ (because $g^{-1}=-g$). Set
$${\rm dist}(g_1,g_2)=\rho(g_1^{-1}g_2),\quad g_1,g_2\in G.$$
Clearly, ${\rm dist}$ is translation-invariant in the sense that
$${\rm dist}(gg_1,gg_2)={\rm dist}(g_1,g_2),\quad g_1,g_2,g\in G.$$
This is a quasi-metric on $G.$ Indeed, there exists a constant $c_G$ such that
$$\rho(g_1g_2)\leq c_G(\rho(g_1)+\rho(g_2)).$$
Indeed, we have
\[
    p_{k}(g_1g_2) = \sum_{j=1}^k p_{k-j}(g_1)p_{j}(g_2)
\]
and therefore
\[
    \|p_k(g_1g_2)\|_{\infty}^{\frac1k} \leq \sum_{j=1}^k \|p_{k-j}(g_1)\|_\infty^{\frac1k} \|p_j(g_2)\|_{\infty}^{\frac1k} \leq \sum_{j=1}^k \rho(g_1)^{\frac{k-j}{k}}\rho(g_2)^{\frac{j}{k}}.
\]
By the convexity of the exponential function we have
\[
    \rho(g_1)^{\frac{k-j}{k}}\rho(g_2)^{\frac{j}{k}} \leq \frac{k-j}{k}\rho(g_1)+\frac{j}{k}\rho(g_2)
\]
Therefore
\[
    \rho(g_1g_2) \leq \max_{1\leq k\leq n} \sum_{j=1}^k \frac{k-j}{k}\rho(g_1)+ \frac{j}{k}\rho(g_2) \leq n(\rho(g_1)+\rho(g_2)).
\]
This inequality implies that ${\rm dist}$ is a quasi-metric. That is,
\begin{align*}
    {\rm dist}(g_1,g_3)&=\rho(g_1^{-1}g_3)=\rho(g_1^{-1}g_2\cdot g_2^{-1}g_3)\\
                       &\leq c_G(\rho(g_1^{-1}g_2)+\rho(g_2^{-1}g_3))=c_G({\rm dist}(g_1,g_2)+{\rm dist}(g_2,g_3)).
\end{align*}
We denote by $B(x,r)\subset G$ the ball in $G$ defined by the quasi-metric $\rho.$ That is, $B(x,r) = \{g \in G\;:\;{\rm dist}(x,g)< r\}.$ 
Since every ball in the metric defined by $\rho$ contains a Euclidean ball, and vice versa, the Haar measure assigns positive and finite measure to every ball in the metric defined by $\rho.$

\begin{lemma}\label{locally_finite_coverings_exist}
Let $G$ be a stratified Lie group. There exists a bounded open subset $U$ of $G$ and a countable subset $\Gamma$ of $G$ such that
$$G = \bigcup_{\gamma\in \Gamma} \gamma U$$
and such that
$$\sup_{x\in G}|\{\gamma\in\Gamma:\ x\in \gamma U\}|<\infty.$$
\end{lemma}
\begin{proof} Let
$$D=\{A\subset G:\ {\rm dist}(g_1,g_2)\geq 1\mbox{ whenever }g_1,g_2\in A\}.$$

We claim that every chain in $D$ has a supremum. Indeed, let $\{A_i\}_{i\in\mathbb{I}}$ be a chain in $D$ and let $A=\cup_{i\in\mathbb{I}}A_i.$ Let $g_1,g_2\in A.$ It follows that $g_1\in A_{i_1},$ $g_2\in A_{i_2}$ for some $i_1,i_2\in\mathbb{I}.$ Since $\mathbb{I}$ is a directed set, it follows that there exists $i_3\in\mathbb{I}$ such that $A_{i_1},A_{i_2}\subset A_{i_3}.$ Hence, $g_1,g_2\in A_{i_3}$ and, hence, ${\rm dist}(g_1,g_2)\geq 1.$ Thus, $A\in D.$ 
	
By Zorn's lemma, there exists at least one maximal element of $D.$
	
If $A\in D$ is a maximal element, then we claim that $\{B(g,1)\}_{g\in A}$ covers $G$ with multiplicity at most 
$$c_G'=\frac{\nu(B(0,2c_G))}{\nu(B(0,\frac1{2c_G}))}.$$

Let $g\in G.$ Consider $A\cup\{g\}.$ Since $A$ is maximal, it follows that either $A\cup\{g\}=A$ or $A\cup\{g\}\notin D.$ In the first case, we have
$$g\in B(g,1)\subset \cup_{g'\in A}B(g',1).$$
In the second case, there exists $g'\in G.$ such that ${\rm dist}(g,g')<1.$ It follows that
$$g\in B(g',1)\subset \cup_{g'\in A}B(g',1).$$
Hence, $\{B(g,1)\}_{g\in A}$ covers $G.$
		
Let $g\in G$ and let
$$n(g)=|\{g'\in A:\ {\rm dist}(g,g')<1\}|.$$
Hence, there exist $\{g_k\}_{k=1}^{n(g)}\subset A$ such that ${\rm dist}(g,g_k)<1,$ $1\leq k\leq n(g).$ If $1\leq k_1,k_2\leq n(g)$ are distinct, then ${\rm dist}(g_{k_1},g_{k_2})\geq 1.$ Using quasi-triangle inequality, we obtain
$$B(g_{k_1},\frac1{2c_G})\cap B(g_{k_2},\frac1{2c_G})=\varnothing.$$
Thus,
$$\nu\Big(\cup_{k=1}^{n(g)}B(g_k,\frac1{2c_G})\Big)=n(g)\cdot \nu(B(0,\frac1{2c_G})).$$
On the other hand, it follows from quasi-triangle inequality that
$$B(g_k,\frac1{2c_G})\subset B(g_k,1)\subset B(g,2c_G).$$
Thus,
$$\nu\Big(\cup_{k=1}^{n(g)}B(g_k,\frac1{2c_G})\Big)\leq \nu(B(0,2c_G)).$$
Combining these inequalities, we obtain
$$n(g)\cdot \nu(B(0,\frac1{2c_G}))\leq \nu(B(0,2c_G)).$$
This provides a required bound on $n(g).$
\end{proof}

A pair $(U,\Gamma)$ as in Lemma \ref{locally_finite_coverings_exist} is called a {\it covering of bounded multiplicity}. We remind the reader that it is not necessary that $\Gamma$ be a subgroup of $G.$

\begin{definition} Let $(\Gamma,U)$ be a covering of bounded multiplicity as in Lemma \ref{locally_finite_coverings_exist}. For $p> 0$ and $q>0$ define
$$\|f\|_{\ell_p(L_q)} := \left(\sum_{\gamma\in \Gamma} \|f\|_{L_q(\gamma U)}^p\right)^{\frac1p}.$$
and
$$\ell_p(L_q)(G) := \{f \in L_{q,\loc}(G)\;:\; \|f\|_{\ell_p(L_q)} < \infty\}.$$
\end{definition}

\begin{definition} Define a sequence space $\ell_{p,\log}$ by setting
$$\|a\|_{p,\log}=\left(\sum_{n\geq0}\log(n+2)\mu(n,a)^p\right)^{\frac1p},$$
$$\ell_{p,\log}=\{a\in l_{\infty}:\ \|a\|_{p,\log}<\infty\}.$$
Here, $\{\mu(n,a)\}_{n=0}^\infty$ is the non-increasing rearrangement of the sequence $|a|.$
\end{definition}

\begin{definition} Let $(\Gamma,U)$ be a covering of bounded multiplicity as in Lemma \ref{locally_finite_coverings_exist}. For $p> 0$ and $q>0$ define
$$\|f\|_{\ell_{p,\log}(L_q)(G)} :=\Big\|\{\|f\|_{L_q(\gamma U)}\}_{\gamma\in\Gamma}\Big\|_{p,\log}.$$
and
$$\ell_{p,\log}(L_q)(G) := \{f \in L_{q,\loc}(G)\;:\; \|f\|_{\ell_{p,\log}(L_q)} < \infty\}.$$
\end{definition}

As indicated by the notation, the space $\ell_{p}(L_q)(G)$ is independent of the choice of covering $(\Gamma,U).$ We prove this in Lemma \ref{lpL_q independence lemma}. Similarly, $\ell_{p,\log}(L_q)(G)$ is independent of the choice of covering $(\Gamma,U),$ but we omit the proof.

To demonstrate this, we use the following result. The translation-invariance of the measure $\nu$ is used in an essential way.
\begin{lemma}\label{fgamma cardinality lemma} Let $(\Gamma,U)$ and $(\Gamma',U')$ be coverings of bounded multiplicity. For every $\gamma\in\Gamma,$ set
$$F_{\gamma}=\{\gamma'\in\Gamma': \gamma U\cap \gamma'U'\neq\varnothing\}.$$
We have
$$\sup_{\gamma\in\Gamma}|F_{\gamma}|<\infty.$$
\end{lemma}
\begin{proof} Fix $x_0\in\gamma U.$ If $\gamma U\cap \gamma' U'\neq\varnothing,$ then there exists $y\in \gamma U\cap \gamma' U'\neq\varnothing.$ Since $x_0,y\in \gamma U,$ it follows that ${\rm dist}(x,y)\leq {\rm diam}(U).$ If $z\in \gamma'U',$ then ${\rm dist}(y,z)\leq {\rm diam}(U').$ Thus,
$${\rm dist}(x_0,z)\leq c_G({\rm diam}(U)+{\rm diam}(U'))\stackrel{def}{=}r_{U,U'}.$$
Hence,
$$\gamma'U'\subset B(x_0,r_{U,U'}),\quad\gamma'\in F_{\gamma}.$$
In other words,
$$A\subset B(x_0,r_{U,U'}),\quad A=\cup_{\gamma'\in F_{\gamma}}\gamma'U'.$$

Let 
$$n'(x)=|\{\gamma'\in\Gamma':\ x\in \gamma' U'\}|\mbox{ and let }n'=\sup_{x\in\Gamma'}n'(x)<\infty.$$
We have
$$\sum_{\gamma'\in F_{\gamma}}\nu(U')=\sum_{\gamma'\in F_{\gamma}}\int_A\chi_{\gamma'U'}=\int_A\sum_{\gamma'\in F_{\gamma}}\chi_{\gamma'U'}\leq\int_A\sum_{\gamma'\in\Gamma'}\chi_{\gamma'U'}.$$
For every $x\in G,$ we have
$$\sum_{\gamma'\in\Gamma'}\chi_{\gamma'U'}=n'(x)\leq n'.$$
Thus,
$$\sum_{\gamma'\in F_{\gamma}}\nu(U')\leq \int_A n'=n'\nu(A)\leq n'\nu(B(x_0,r_{U,U'}))=n'\nu(B(1_G,r_{U,U'})).$$
In other words,
$$|F_{\gamma}|\leq n'\frac{\nu(B(1_G,r_{U,U'}))}{\nu(U')}.$$
\end{proof}

\begin{lemma}\label{lpL_q independence lemma}
The definition of $\ell_p(L_q)(G)$ is independent of the choice of covering $(\Gamma,U),$ in the sense that if $(\Gamma,U)$ and $(\Gamma',U')$ are coverings of bounded multiplicity there exist constants $c,C>0$ (depending on the coverings) such that
    \[
        c\left(\sum_{\gamma\in \Gamma'} \|f\|_{L_q(\gamma U')}^p\right)^{1/p} \leq \left(\sum_{\gamma\in \Gamma} \|f\|_{L_q(\gamma U)}^p\right)^{1/p}\leq C\left(\sum_{\gamma\in \Gamma'} \|f\|_{L_q(\gamma U')}^p\right)^{1/p}.
    \]
\end{lemma}
\begin{proof} By symmetry, it suffices to prove the upper bound. We have
$$\|f\|_{L_q(\gamma U)}=\|f\chi_{\gamma U}\|_q.$$
By the definition of $F_{\gamma}$ in Lemma \ref{fgamma cardinality lemma}, we have
$$\chi_{\gamma U}\leq \sum_{\gamma'\in F_{\gamma}}\chi_{\gamma'U'}.$$
If $q\geq1,$ then it follows from the triangle inequality that
$$\|f\|_{L_q(\gamma U)}\leq \|\sum_{\gamma'\in F_{\gamma}}f\chi_{\gamma'U'}\|_q\leq\sum_{\gamma'\in F_{\gamma}}\|f\chi_{\gamma'U'}\|_q=\sum_{\gamma'\in F_{\gamma}}\|f\|_{L_q(\gamma'U')}.$$
If $q\in(0,1),$ then it follows from the $q$-convexity of the $L_q$ quasi-norm that
$$\|f\|_{L_q(\gamma U)}^q\leq \|\sum_{\gamma'\in F_{\gamma}}f\chi_{\gamma'U'}\|_q^q\leq\sum_{\gamma'\in F_{\gamma}}\|f\chi_{\gamma'U'}\|_q^q=\sum_{\gamma'\in F_{\gamma}}\|f\|_{L_q(\gamma'U')}^q.$$	
In either case, it follows from Lemma \ref{fgamma cardinality lemma} that
$$\|f\|_{L_q(\gamma U)}^q\leq c_{p,q,U,U'}\sum_{\gamma'\in F_{\gamma}}\|f\|_{L_q(\gamma'U')}^p.$$
Thus,
$$\sum_{\gamma\in\Gamma}\|f\|_{L_q(\gamma U)}^q\leq c_{p,q,U'U'} \sum_{\gamma\in\Gamma}\sum_{\gamma'\in F_{\gamma}}\|f\|_{L_q(\gamma'U')}^p=c_{p,q,U,U'}\sum_{\gamma'\in\Gamma'}\|f\|_{L_q(\gamma'U')}^p\sum_{\substack{\gamma\in\Gamma\\ \gamma'\in F_{\gamma}}}1.$$

Let
$$F_{\gamma'}=\{\gamma\in\Gamma:\ \gamma U\cap\gamma'U'\neq\varnothing\}.$$
It is immediate that $\gamma'\in F_{\gamma}$ iff $\gamma\in F_{\gamma'}.$ Thus,
$$\sum_{\substack{\gamma\in\Gamma\\ \gamma'\in F_{\gamma}}}1=|F_{\gamma'}|.$$
By the preceding paragraph and Lemma \ref{fgamma cardinality lemma}, we have
$$\sum_{\gamma\in\Gamma}\|f\|_{L_q(\gamma U)}^q\leq c_{p,q,U,U'}\sum_{\gamma'\unboldmath\Gamma'}\|f\|_{L_q(\gamma'U')}^p|F_{\gamma'}|\leq c_{p,q,U,U'}\cdot c_{U',U}\sum_{\gamma'\unboldmath\Gamma'}\|f\|_{L_q(\gamma'U')}^p.$$
\end{proof}
By an identical argument, we may also prove that the space $\ell_{p,\log}(L_q)(G)$ is also independent of the choice of $(\Gamma,U).$

\subsection{Cwikel estimate in $\mathcal{L}_p$ and $\Lc_{p,\infty}$}
We now explain how to remove the assumption of compact support from Lemma \ref{less_simple_cwikel}. This will be a consequence of some operator inequalities, in a similar manner to
\cite{LeSZ-cwikel}. The following result is an immediate combination of \cite[Lemma 3.3.7]{LSZ} and \cite[Proposition 2.7]{LeSZ-cwikel}. We provide a short argument using the terminology
of majorisation and direct sum. For further details see \cite{LeSZ-cwikel}.
\begin{lemma}\label{direct_sum_lemma} Let $0<p<2$ and let $\{T_n\}_{n\geq0}$ be such that
$$T_{n_1}^{\ast}T_{n_2}=0,\quad n_1\neq n_2.$$
We have
$$\|\sum_{n\geq0}T_n\|_p\leq(\sum_{n\geq0}\|T_n\|_p^p)^\frac1p,\quad \|\sum_{n\geq0}T_n\|_{p,\infty}\leq c_p(\sum_{n\geq0}\|T_n\|_{p,\infty}^p)^{\frac1p}.$$
Here, the series in the left hand side converge in in $\mathcal{L}_p$ (respectively, in $\mathcal{L}_{p,\infty}$) provided that numerical series on the right hand side converge.
\end{lemma}
\begin{proof} 
Denote for brevity $T=\sum_{n\geq0}T_n.$ We have
$$|T|^2=\sum_{n_1,n_2\geq0}T_{n_1}^{\ast}T_{n_2}=\sum_{n\geq0}|T_n|^2.$$
Using \cite[Lemma 3.3.7]{LSZ}, we have
$$\bigoplus_{n\geq0}|T_n|^2\prec\sum_{n\geq0}|T_n|^2.$$
From \cite[Proposition 2.7]{LeSZ-cwikel}, we obtain
$$\|T\|_{p,\infty}^2=\||T|^2\|_{\frac{p}{2},\infty}\leq c_p\|\bigoplus_{n\geq0}|T_n|^2\|_{\frac{p}{2},\infty}=\|\bigoplus_{n\geq0}T_n\|_{p,\infty}^2.$$
The assertion follows from another standard inequality
$$\|\bigoplus_{n\geq0}T_n\|_{p,\infty}\leq(\sum_{n\geq0}\|T_n\|_{p,\infty}^p)^{\frac1p}.$$
\end{proof}

\begin{proof}[Proof of Theorem \ref{main_nontrivial_cwikel_theorem}.\eqref{mnctb}]
Let $(\Gamma,U)$ be a covering with bounded multiplicity for $G$. Since $\Gamma$ is countable, we choose an enumeration $\Gamma=\{\gamma_k\}_{k\geq0}.$ Set
$$A_{-1}=\varnothing,\quad A_k=(\gamma_nU)\backslash\cup_{l<k}\gamma_lU.$$
Set
$$T_k=M_{f\chi_{A_k}}(1-\Delta)^{-\frac{d_{\rm hom}}{2p}},\quad n\geq0.$$
It is immediate that
$$\sum_{k\geq0}T_k=M_f(1-\Delta)^{-\frac{d_{\rm hom}}{2p}}.$$

By Lemma \ref{direct_sum_lemma}, we have
\begin{align*}
    \|M_f(1-\Delta)^{-\frac{d_{\rm hom}}{2p}}\|_{p,\infty}&\leq c_p^{\frac12}(\sum_{k\geq0}\|M_{f\chi_{A_k}}(1-\Delta)^{-\frac{d_{\rm hom}}{2p}}\|_{p,\infty}^p)^{\frac1p}\\
                                                          &\leq c_p^{\frac12}(\sum_{k\geq0}\|M_{f\chi_{\gamma_kU}}(1-\Delta)^{-\frac{d_{\rm hom}}{2p}}\|_{p,\infty}^p)^{\frac1p}\\
                                                          &=c_p^{\frac12}(\sum_{k\geq0}\|M_{T_{\gamma_k}f\cdot \chi_U}(1-\Delta)^{-\frac{d_{\rm hom}}{2p}}\|_{p,\infty}^p)^{\frac1p}.
\end{align*}
Here, $T_{\gamma_k}f$ is the shift of $f$ by $\gamma_k.$ The assertion follows now from Lemma \ref{less_simple_cwikel}.
\end{proof}

\begin{proof}[Proof of Theorem \ref{main_nontrivial_cwikel_theorem}.\eqref{mnctc}] It is immediate that
$$\Big\|M_f(1-\Delta)^{-\frac{d_{\hom}}{4}}\|_{2,\infty}^2=\|(1-\Delta)^{-\frac{d_{\hom}}{4}}M_{|f|^2}(1-\Delta)^{-\frac{d_{\hom}}{4}}\|_{1,\infty}.$$
Let $(\gamma_k)_{k\geq1}$ be enumeration of $\Gamma$ such that the sequence
$$\{\|f\|_{L_q(\gamma_kU)}\}_{k\geq1}$$
is decreasing. We have
$$1\leq \sum_{k\geq1}\chi_{\gamma_k U}.$$
Thus,
$$(1-\Delta)^{-\frac{d_{\hom}}{4}}M_{|f|^2}(1-\Delta)^{-\frac{d_{\hom}}{4}}\leq \sum_{k\geq1}(1-\Delta)^{-\frac{d_{\hom}}{4}}M_{|f|^2\chi_{\gamma_k U}}(1-\Delta)^{-\frac{d_{\hom}}{4}}.$$
In particular, we have
$$\Big\|M_f(1-\Delta)^{-\frac{d_{\hom}}{4}}\|_{2,\infty}^2\leq\| \sum_{k\geq1}(1-\Delta)^{-\frac{d_{\hom}}{4}}M_{|f|^2\chi_{\gamma_k U}}(1-\Delta)^{-\frac{d_{\hom}}{4}}\|_{1,\infty}.$$
By Proposition 5.4 in \cite{LeSZ-cwikel}, we  have the following replacement for triangle inequality in $\mathcal{L}_{1,\infty}:$
$$\|\sum_{k\geq1}T_k\|_{1,\infty}\leq 4\sum_{k\geq1}(1+\log(k))\|T_k\|_{1,\infty}.$$
Taking into account the obvious inequality
$$4(1+\log(k))\leq 12\log(k+1),\quad k\geq1,$$
we write
\begin{align*}
    \Big\|M_f(1-\Delta)^{-\frac{d_{\hom}}{4}}&\|_{2,\infty}^2\leq 12 \sum_{k\geq1}\log(k+2)\|(1-\Delta)^{-\frac{d_{\hom}}{4}}M_{|f|^2\chi_{\gamma_k U}}(1-\Delta)^{-\frac{d_{\hom}}{4}}\|_{1,\infty}\\
                                             &=12\sum_{k\geq1}\log(k+2)\|(1-\Delta)^{-\frac{d_{\hom}}{4}}M_{T_{\gamma_k}(|f|^2\chi_{\gamma_k U})}(1-\Delta)^{-\frac{d_{\hom}}{4}}\|_{1,\infty}.
\end{align*}
Here, $T_{\gamma_g}f$ is the shift of $f$ by $\gamma_k.$ It follows now from Lemma \ref{less_simple_cwikel} that
\begin{align*}
\Big\|M_f(1-\Delta)^{-\frac{d_{\hom}}{4}}&\|_{2,\infty}^2 \leq c_{q,U}\sum_{k\geq1}\log(k+2)\|T_{\gamma_k}(|f|\chi_{\gamma_k U})\|_{L_q(U)}^2\\
                                         &=c_{q,U}\sum_{k\geq1}\|f\|_{L_q(\gamma_kU)}^2=c_U\|f\|_{\ell_{2,\log}}^2.
\end{align*}
Here, the last inequality is due to the particular choice of the enumeration $\gamma.$
\end{proof}

\section{Spectral asymptotic formulae}
In this section we will prove the asymptotic formula in Theorem \ref{main_asymptotic_formula}. The proof will be performed using the Wiener-Ikehara Tauberian theorem and certain
operator estimates from \cite{SZ-asterisque}. 
It will be helpful to refer to the following

The basic tool we use will be the following:
\begin{lemma}\label{trace_formula}
    Let $f \in \ell_1(L_q)(G)$ and $g \in C_0(\Rl_+)$ be such that $t\mapsto (1+t)^N|g(t)| \in C_0(\Rl_+)$ for some $N > \frac{d_{\hom}}{2}.$
    Then
    \[
        \Tr(M_fg(-\Delta)) = \left(\int_{G} f\right) \tau(g(-\Delta)).
    \]
\end{lemma}
\begin{proof}
    Note that the assumptions on $f$ and $g$ imply that $M_fg(-\Delta)$ belongs to $\Lc_1(L_2(G)).$ Indeed, from Theorem \ref{main_nontrivial_cwikel_theorem}, when $N>\frac{d_{\hom}}{2}$ we have
    \[
        M_f(1-\Delta)^{-N} \in \Lc_1
    \]
    and thus
    \[
        M_fg(-\Delta) = M_f(1-\Delta)^{-N}\cdot (1-\Delta)^Ng(-\Delta)
    \]
    Functional calculus and the assumption on $g$ implies that there exists $N>\frac{d_{\hom}}{2}$ such that $(1-\Delta)^{N}g(-\Delta)$ is bounded,
    and thus $M_fg(-\Delta) \in \Lc_1.$ Moreover, we have the estimate
    \[
        \|M_fg(-\Delta)\|_{1} \leq C_{G}\|f\|_{\ell_1(L_q)(G)}\sup_{t>0} (1+t)^N |g(t)|.
    \]
    The Stone-Weierstrass theorem implies that for every $g \in C_0(\Rl_+)$ such that $t\mapsto (1+t)^{N} |g(t)|\in C_0(\Rl_+),$ there exists
    a sequence $\{g_n\}_{n=0}^\infty \subseteq \mathrm{span}(\{t\mapsto e^{-st}\}_{s >0})$ such that 
    \[
        \lim_{n\to\infty} \sup_{t>0} (1+t)^N |g_n(t)-g(t)| < \infty.
    \]    
    Since $C^\infty_c(G)$ is dense in $\ell_1(L_q)(G),$ it suffices to verify the formula for $f \in C^\infty_c(G)$ and $g(t) = e^{-st}$ for some $s>0.$
    
    To this end, $g(-\Delta) = e^{s\Delta}$ is equal to $\lambda(h_s),$ where $h_s\in C^\infty(G)\cap L_1(G)$ \cite[Theorem 4.2.7]{FischerRuzhansky2016} and $M_fg(-\Delta)$ has kernel function
    \[
        (\gamma,\eta)\mapsto f(\gamma)h_s(\gamma \eta^{-1}),\quad \gamma,\eta \in G.
    \]
    Since $f$ is compactly supported and the kernel is everywhere smooth, the trace of this operator is given by the integral along the diagonal $\gamma=\eta,$
    which yields
    \[
        \Tr(M_fg(-\Delta)) = \left(\int_{G} f\right) h_s(1).
    \]
    By definition, $\tau(g(-\Delta)) = h_s(1),$ and this completes the proof.    
\end{proof}

\begin{corollary}\label{specific_trace_formula}
    Let $z \in \Cplx$ such that $\Re(z)>d_{\hom}$ and let $0\leq f \in C^\infty_c(G).$ There exists a constant $c_G>0$ such that
    \[
        \Tr(M_f^{2z}(1-\Delta)^{-\frac{z}{2}}) = c_G\frac{\Gamma(\frac{z-d_{\hom}}{2})}{\Gamma(\frac{z}{2})}\int_{G} f^{2z}.
    \]
\end{corollary}
\begin{proof}
By Corollary \ref{christ_trace_formula}, the substitution $s = \frac{t}{1+t}$ yields
\begin{align*}
    \tau((1-\Delta)^{-\frac{z}{2}})&=c_G\int_0^\infty (1+t)^{-\frac{z}{2}}t^{\frac{d_{\hom}}{2}-1}dt\\
                                   &= c_G\int_0^1 (1-s)^{\frac{z-d_{\hom}}{2}-1}s^{\frac{d_{\hom}}{2}-1}\,ds = c_G\mathrm{B}(\frac{z-d_{\hom}}{2},\frac{d_{\hom}}{2}).
\end{align*}
Here, ${\rm B}$ is the Beta function.
The assertion follows now from Corollary \ref{christ_trace_formula}, Lemma \ref{trace_formula} and the standard identity
$$\mathrm{B}(\frac{z-d_{\hom}}{2},\frac{d_{\hom}}{2}) =\frac{\Gamma(\frac{z-d_{\hom}}{2})\Gamma(\frac{d_{\hom}}{2})}{\Gamma(\frac{z}{2})}.$$   
\end{proof}
Note that the constant in the above corollary differs from that in Corollary \ref{christ_trace_formula} by a factor of $\Gamma\left(\frac{d_{\hom}}{2}\right).$

\begin{lemma}\label{zeta_function_analytic_continuation}
    The function
    \[
        z\mapsto \Tr(M_f^{2z}(1-\Delta)^{-\frac{z}{2}}),\quad \Re(z)>d_{\hom}
    \]
    admits meromorphic continuation to the half-plane $\Re(z)>d_{\hom}-2,$ with only a simple pole
    at the point $z = d_{\hom}$ and
    \[
        \mathrm{Res}_{z=d_{\hom}} \Tr(M_f^{2z}(1-\Delta)^{-\frac{z}{2}}) = c_G\int_{G} f^{2d_{\hom}}.
    \]
\end{lemma}
\begin{proof} Recall that the Gamma function is meromorphic on the complex plane $\Cplx,$ with simple poles located at $\{0,-1,-2,\ldots\},$ and the reciprocal $\frac{1}{\Gamma(z)}$ is entire. Since $f\geq 0,$ the function
\[
    z\mapsto \int_G f^{2z}
\]
is holomorphic in the half-plane $\Re(z)>0.$ Combining these facts, it follows that the function
\[
    z\mapsto \frac{\Gamma(\frac{z-d_{\hom}}{2})}{\Gamma(\frac{z}{2})}\int_{G} f^{2z},\quad \Re(z)>d_{\hom}
\]
admits analytic continuation to the set 
\[
    \{z \in \Cplx\;:\;\Re(z)>0\}\setminus \{z\in\Cplx\;:\; \frac{z-d_{\hom}}{2} = 0,-1,-2,\ldots\} = \{z \in \Cplx\;:\; \Re(z)>0\}\setminus (d_{\hom}-2\Itgr_+).
\]
It follows from the functional equation $\Gamma(z)=\frac{1}{z}\Gamma(z+1)$ that the pole of $\Gamma$ at zero has corresponding residue equal to $\Gamma(1)=1.$ Hence, the analytic continuation of the function
\[
    z\mapsto \frac{\Gamma(\frac{z-d_{\hom}}{2})}{\Gamma(\frac{z}{2})}\int_{G} f^{2z},\quad \Re(z)>d_{\hom}
\]
has a simple pole at $z = d_{\hom}$ with corresponding residue
\[
    \frac{2}{\Gamma(\frac{d_{\hom}}{2})}\int_{G} f^{2d_{\hom}}.
\]
The result now follows from Corollary \ref{specific_trace_formula}
\end{proof}

\begin{lemma}\label{asterisque_conditions}
Let $0\leq f \in C^\infty_c(G)$, and let $A=M_{f}^2,$ and $B=(1-\Delta)^{-\frac 12}.$
We have
\begin{enumerate}[{\rm (i)}]
    \item\label{cond1} $B^{d_{\hom}}A\in\mathcal{L}_{1,\infty};$
    \item\label{cond2} $B^{q-2}[B,A]\in\mathcal{L}_1$ for every $q>d_{\hom};$
    \item\label{cond3} $A^{\frac12}BA^{\frac12}\in\mathcal{L}_{d_{\hom},\infty};$
    \item\label{cond4} $[B,A^{\frac12}]\in\mathcal{L}_{\frac{d_{\hom}}{2},\infty}.$
\end{enumerate}
\end{lemma}
\begin{proof} Theorem \ref{MSX_commutator_estimate}.\eqref{msxc} with $\alpha=d_{\hom}$ and $\gamma=0$ implies \eqref{cond1}. 
    Next we prove \eqref{cond2}. For $q>d_{\hom},$ we have
    \begin{align*}
        B^{q-2}[B,A]&=-(1-\Delta)^{-\frac{q-1}{2}}[(1-\Delta)^{\frac 12},M_{f}^2](1-\Delta)^{-\frac 12}
    \end{align*}
    This belongs to $\Lc_1$ for every $q > d_{\hom}$ by Theorem \ref{MSX_commutator_estimate}.\eqref{msxd}.
    
    To prove \eqref{cond3}, we observe that
    \[
        A^{\frac12}BA^{\frac12} = |M_{f}(1-\Delta)^{-\frac{1}{4}}|^2
    \]
    and Theorem \ref{main_nontrivial_cwikel_theorem}.\eqref{mncta} implies $M_{f}(1-\Delta)^{-\frac{1}{4}} \in \Lc_{2d_{\hom},\infty},$
    and hence \eqref{cond3} follows from H\"older's inequality.
    
    Finally, \eqref{cond4} is the assertion that
    \[
        [(1-\Delta)^{-\frac{1}{2}},M_{f}] \in \Lc_{\frac{d_{\hom}}{2},\infty}.
    \]
    This is a special case of Theorem \ref{MSX_commutator_estimate}, with $\alpha=\gamma=0$ and $\beta=-1.$
\end{proof}

\begin{lemma}\label{wiener_ikehara_tauberian_condition}
Let $A=M_{f}^2,$ $0\leq f\in C^{\infty}_c(G),$ and $B=(1-\Delta)^{-\frac 12}.$ The function
\[
    z\to {\rm Tr}((A^{\frac12}BA^{\frac12})^z),\quad\Re(z)>d_{\hom},
\]
admits a meromorphic continuation to the half-plane $\{\Re(z)>d_{\hom}-1\}.$ The only singularity is the point $z=d_{\hom},$ where the function has a simple pole. We have
\[
    {\rm Res}_{z=d_{\hom}}{\rm Tr}((A^{\frac12}BA^{\frac12})^z)=c_G\int_{G}f^{2d_{\hom}}.
\]
\end{lemma}
\begin{proof} 
It is proved in \cite[Theorem 5.4.2]{SZ-asterisque} that under the conditions of Lemma \ref{asterisque_conditions}, the function
\[
    z\mapsto \Tr((A^{\frac12}BA^{\frac12})^z)-\Tr(A^zB^z),\quad \Re(z) > d_{\hom}
\]
admits analytic continuation to the half-plane $\Re(z)>d_{\hom}-1.$ 

Lemma \ref{zeta_function_analytic_continuation} asserts that
\[
    z\mapsto \Tr(A^zB^z),\quad \Re(z)>d_{\hom}
\]
admits meromorphic continuation to the half-plane $\Re(z)>d_{\hom}-1,$ with only a simple pole at $z=d_{\hom}$ and corresponding residue
\[
    c_G\int_G f^{2d_{\hom}}.
\]
It follows that
\[
    z\mapsto \Tr((A^{\frac12}BA^{\frac12})^z)
\]
also admits meromorphic continuation to the half-plane $\Re(z)>d_{\hom}-1,$ with only a simple pole at $z=d_{\hom}$ and the same residue.
\end{proof}

Recall the following version of Wiener-Ikehara Tauberian theorem \cite[Section 4]{Korevaar-tauberian-2004}.
\begin{theorem}\label{wiener-ikehara theorem} Let $0\leq V\in\mathcal{L}_{p,\infty}$ and let
$$\zeta_V(z)={\rm Tr}(V^z),\quad \Re(z)>p.$$
If there exists $c_V>0$ such that the function
$$z\to\zeta_V(z)-\frac{c_V}{z-p},\quad \Re(z)>p,$$
extends continuously to the closed half-plane $\{\Re(z)\geq p\},$ then there exists a limit
$$\lim_{t\to\infty}t\mu^p(t,V)=\frac{c_V}{p}.$$
\end{theorem}

\begin{lemma}\label{pseudodifferential_assertion} Let $f \in C^\infty_c(G).$ For all $k\geq 1$ we have	
$$(1-\Delta)^{-\frac{k}{4}}M_{f^{2k}}(1-\Delta)^{-\frac{k}{4}}-((1-\Delta)^{-\frac14}M_{f^2}(1-\Delta)^{-\frac14})^k\in\Lc_{\frac{d_{\hom}}{k+1},\infty}.$$	
\end{lemma}
\begin{proof} Let us prove the first assertion by induction on $k.$ Base of induction (the case $k=1$) is trivial. It remains to prove the step of induction.
	
Suppose the first asserion holds for $k.$ Let us prove it for $k+1.$ By H\"older inequality, we have
\begin{align*}
(1-\Delta)^{-\frac{k}{4}}&M_{f^{2k}}(1-\Delta)^{-\frac{k+1}{4}}M_{f^2}(1-\Delta)^{-\frac14}-((1-\Delta)^{-\frac14}M_{f^2}(1-\Delta)^{-\frac14})^{k+1}\\
&=\Big((1-\Delta)^{-\frac{k}{4}}M_{f^{2k}}(1-\Delta)^{-\frac{k}{4}}-((1-\Delta)^{-\frac14}M_{f^2}(1-\Delta)^{-\frac14})^k\Big)\cdot (1-\Delta)^{-\frac14}M_{f^2}(1-\Delta)^{-\frac14}\\
&\in\Lc_{\frac{d_{\hom}}{k+1},\infty}\cdot \Lc_{d_{\hom},\infty}\subset \Lc_{\frac{d_{\hom}}{k+2},\infty}.
\end{align*}
On the other hand, it follows from Theorem \ref{MSX_commutator_estimate} that
\begin{align*}
&(1-\Delta)^{-\frac{k}{4}}M_{f^{2k}}(1-\Delta)^{-\frac{k+1}{4}}M_{f^2}(1-\Delta)^{-\frac14}-(1-\Delta)^{-\frac{k+1}{4}}M_{f^{2k+2}}(1-\Delta)^{-\frac{k+1}{4}}\\
&=(1-\Delta)^{-\frac{k+1}{4}}[(1-\Delta)^{\frac14},M_{f^{2k}}]\cdot (1-\Delta)^{-\frac{k+1}{4}}M_{f^2}(1-\Delta)^{-\frac14}\\
&\quad +(1-\Delta)^{-\frac{k+1}{4}}M_{f^{2k}}\cdot [(1-\Delta)^{-\frac{k}{4}},M_{f^2}](1-\Delta)^{-\frac14}\\
&\in \mathcal{L}_{\frac{2d_{\hom}}{k+2},\infty}\cdot\mathcal{L}_{\frac{2d_{\hom}}{k+2},\infty}+\mathcal{L}_{\frac{2d_{\hom}}{k+3},\infty}\subset\Lc_{\frac{d_{\hom}}{k+2},\infty}.
\end{align*}
Combining these $2$ inclusions, we obtain
$$(1-\Delta)^{-\frac{k+1}{4}}M_{f^{2k+2}}(1-\Delta)^{-\frac{k+1}{4}}-((1-\Delta)^{-\frac14}M_{f^2}(1-\Delta)^{-\frac14})^{k+1}\in\Lc_{\frac{d_{\hom}}{k+2},\infty}.$$	
This establishes the step of induction and, therefore, proves the lemma.
\end{proof}

Having established the analytic continuation in Lemma \ref{wiener_ikehara_tauberian_condition}, we apply the Wiener-Ikehara theorem to
deduce the following special case of Theorem \ref{main_asymptotic_formula}.

\begin{lemma}\label{initial_weyl_lemma}
Let $k\in\mathbb{N}$ and let $0 \leq f \in C^\infty_c(G).$ We have
$$\lim_{t\to\infty} t\mu(t,(1-\Delta)^{-\frac{k}{4}}M_{f^{2k}}(1-\Delta)^{-\frac{k}{4}})^{\frac{d_{\hom}}{k}} = c_G\int_G f^{2d_{\hom}}.$$
\end{lemma}
\begin{proof} Let $p=d_{\hom},$ $A=M_{f^2}$ and $B=(1-\Delta)^{-\frac12}.$ Set
$$V=A^{\frac12}BA^{\frac12}=M_f(1-\Delta)^{-\frac12}M_f.$$
Note from Lemma \ref{wiener_ikehara_tauberian_condition} that the assumptions of Theorem \ref{wiener-ikehara theorem} are satisfied for $V.$ Applying Theorem \ref{wiener-ikehara theorem}, we obtain
$$\lim_{t\to\infty} t\mu(t,V)^{d_{\hom}} = c_G\int_{G} f^{2d_{\hom}}.$$
Since
$$\mu(V)=\mu^2(M_f(1-\Delta)^{-\frac14})=\mu((1-\Delta)^{-\frac{1}{4}}M_f^2(1-\Delta)^{-\frac{1}{4}}),$$
it follows that
$$\lim_{t\to\infty} t\mu(t,(1-\Delta)^{-\frac14}M_{f^2}(1-\Delta)^{-\frac14})^{d_{\hom}}= c_G\int_G f^{2d_{\hom}}.$$
Consequently,
$$\lim_{t\to\infty} t\mu(t,((1-\Delta)^{-\frac14}M_{f^2}(1-\Delta)^{-\frac14})^k)^{\frac{d_{\hom}}{k}}= c_G\int_G f^{2d_{\hom}}.$$
By Lemma \ref{pseudodifferential_assertion}, we have
$$(1-\Delta)^{-\frac{k}{4}}M_{f^{2k}}(1-\Delta)^{-\frac{k}{4}}-((1-\Delta)^{-\frac14}M_{f^2}(1-\Delta)^{-\frac14})^k\in(\Lc_{\frac{d_{\hom}}{k},\infty})_0.$$
Therefore, Proposition \ref{elementary_perturbation} yields
$$\lim_{t\to\infty} t\mu(t,(1-\Delta)^{-\frac{k}{4}}M_{f^{2k}}(1-\Delta)^{-\frac{k}{4}})^{\frac{d_{\hom}}{k}} = c_G\int_G f^{2d_{\hom}}.$$
\end{proof}

The next lemma is similar to Lemma \ref{initial_weyl_lemma}, but the operator $(1-\Delta)^{-\frac{k}{4}}$ is replaced with $(-\Delta)^{-\frac{k}{4}}.$ Since $(-\Delta)^{-\frac{k}{4}}$ is unbounded, it is not obvious how to make sense of the operator
\[
    (-\Delta)^{-\frac{k}{4}}M_{f^{2k}}(-\Delta)^{-\frac{k}{4}}.
\]
We will interpret the above expression in the following way. For $0 \in f \in C^\infty_c(G),$ $T := M_{f^{k}}(-\Delta)^{-\frac{k}{4}}$ is a well-defined element of $\Lc_{\frac{2d_{\hom}}{k},\infty}$ for $k < d_{\hom},$ by Theorem \ref{general_cwikel_theorem}. We then define
\[
    (-\Delta)^{-\frac{k}{4}}M_{f^{2k}}(-\Delta)^{-\frac{k}{4}} := T^*T \in \Lc_{\frac{d_{\hom}}{k},\infty}.
\]
\begin{lemma}\label{second_weyl_lemma}
Let $1\leq k<d_{\hom}$ and let $0 \leq f \in C^\infty_c(G).$ We have
$$\lim_{t\to\infty} t\mu(t,(-\Delta)^{-\frac{k}{4}}M_{f^{2k}}(-\Delta)^{-\frac{k}{4}})^{\frac{d_{\hom}}{k}} = c_G\int_G f^{2d_{\hom}}.$$
\end{lemma}
\begin{proof} Applying Theorem \ref{general_cwikel_theorem} to the space $E=L_{\frac{2d_{\hom}}{k}}+L_2$ (this space is an interpolation space for the couple $(L_2,L_{\infty})$), we obtain
\begin{align*}
    \|&M_{f^k}(-\Delta)^{-\frac{k}{4}}-M_{f^k}(1-\Delta)^{-\frac{k}{4}}\|_{\frac{2d_{\hom}}{k}}\\
      &\leq c_{G,E,k}\|f^k\otimes ((-\Delta)^{-\frac{k}{4}}-(1-\Delta)^{-\frac{k}{4}})\|_{(L_{\frac{2d_{\hom}}{k}}+L_2)(L_{\infty}(G)\bar{\otimes}{\rm VN}(G),\tau)}.
\end{align*}
For every $T\in (L_{\frac{2d_{\hom}}{k}}+L_2)({\rm VN}(G),\tau)$ and for every $f\in C_c(G),$ we have (denoting $m$ for the Lebesgue measure on $G$),
\begin{align*}
    \|&f^k\otimes T\|_{(L_{\frac{2d_{\hom}}{k}}+L_2)(L_{\infty}(G)\bar{\otimes}{\rm VN}(G))}\\
      &\leq\|f\|_{\infty}^k\cdot\max\{m({\rm supp}(f))^{\frac{k}{2d_{\hom}}},m({\rm supp}(f))^{\frac{1}{2}}\}\cdot\|T\|_{(L_p+L_q)({\rm VN}(G))}.
\end{align*}

Since $k < d_{\hom},$ we have
\[
    \int_0^1(t^{-\frac{k}{4}}-(t+1)^{-\frac{k}{4}})^2\cdot t^{\frac{d_{\hom}}{2}-1}dt\leq\int_0^1(t^{-\frac{k}{4}})^2\cdot t^{\frac{d_{\hom}}{2}-1}dt=\int_0^1t^{\frac{d_{\hom}-k}{2}-1}dt<\infty,
\]
and
\begin{align*}
\int_1^{\infty}(t^{-\frac{k}{4}}-(t+1)^{-\frac{k}{4}})^{\frac{2d_{\hom}}{k}}\cdot t^{\frac{d_{\hom}}{2}-1}dt&\leq\int_1^{\infty}(\frac{k}{4}t^{-\frac{k}{4}-1})^{\frac{2d_{\hom}}{k}}\cdot t^{\frac{d_{\hom}}{2}-1}dt\\
&=(\frac{k}{4})^{\frac{2d_{\hom}}{k}}\cdot\int_1^{\infty}t^{-\frac{2d_{\hom}}{k}-1}dt<\infty,
\end{align*}
we conclude that
$$(-\Delta)^{-\frac{k}{4}}-(1-\Delta)^{-\frac{k}{4}}\in (L_{\frac{2d_{\hom}}{k}}+L_2)({\rm VN}(G)).$$
It therefore follows from Theorem \ref{general_cwikel_theorem} as stated above that
$$M_{f^k}(-\Delta)^{-\frac{k}{4}}-M_{f^k}(1-\Delta)^{-\frac{k}{4}}\in\mathcal{L}_{\frac{2d_{\hom}}{k}}\subset(\mathcal{L}_{\frac{2d_{\hom}}{k},\infty})_0.$$

Consequently,
$$(-\Delta)^{-\frac{k}{4}}M_{f^{2k}}(-\Delta)^{-\frac{k}{4}}-(1-\Delta)^{-\frac{k}{4}}M_{f^{2k}}(1-\Delta)^{-\frac{k}{4}}\in (\mathcal{L}_{\frac{d_{\hom}}{k},\infty})_0.$$
The assertion follows now from Lemma \ref{initial_weyl_lemma} and Proposition \ref{elementary_perturbation}.
\end{proof}

\begin{proof}[Proof of Theorem \ref{main_asymptotic_formula}] We prove only the first assertion. The proofs of the remaining two identities follow by an identical argument, using Lemma \ref{second_weyl_lemma} in place of Lemma \ref{initial_weyl_lemma} as needed.

Let $1\leq k< d_{\hom},$ and let $0\leq f \in L_{\frac{d_{\hom}}{k}}(G).$ Fix a sequence $\{f_m\}_{m\geq0}\subset C^{\infty}_c(G)$ such that $f_m^{2k}\to f$ in $L_{\frac{d_{\hom}}{k}}(G).$ It follows from Corollary \ref{main_nontrivial_cwikel_corollary} that
$$(1-\Delta)^{-\frac{k}{4}}M_{f_m^{2k}}(1-\Delta)^{-\frac{k}{4}}\to (1-\Delta)^{-\frac{k}{4}}M_f(1-\Delta)^{-\frac{k}{4}},\quad m\to\infty,$$
in $\mathcal{L}_{\frac{d_{\hom}}{k},\infty}.$ The first assertion follows now from Lemma \ref{initial_weyl_lemma} and Proposition \ref{advanced_perturbation}.

\end{proof}

\section{Semiclassical asymptotics}

\subsection{Preliminaries on the Birman-Schwinger principle}

We appeal to the Birman--Schwinger principle, in the form of \cite[Theorem 7.9.4]{Simon-course-IV}, which we briefly recall here. Recall
that a self-adjoint operator $V$ on a Hilbert space $H$ is said to be relatively form-compact with respect to a positive self-adjoint operator $T$ if $\dom(T^{1/2}) \subseteq \dom(|V|^{1/2})$
and the linear operator
\[
    (1+T)^{-\frac12}V(1+T)^{-\frac12}:H\to H
\]
is compact \cite[Definition, p.662]{Simon-course-IV}. It follows that
\[
    (\lambda+T)^{-\frac12}V(\lambda+T)^{-\frac12}:H\to H
\]
is also compact for every $\lambda>0$ \cite[Remark, p.663]{Simon-course-IV}.

The fact that $V$ is form compact relative to $T$ implies that the quadratic form sum $T+V$ is a well-defined lower bounded self-adjoint operator, and for any $\lambda>0$
the part of the spectrum of $T+V$ in $(-\infty,-\lambda)$ consists of at most finitely many eigenvalues. The number of eigenvalues is counted by the Birman--Schwinger principle,
recalled here as follows:
\begin{theorem}[Birman--Schwinger principle]\label{bs_principle}
Let $T$ be a self-adjoint positive unbounded linear operator on a Hilbert space $H,$ and let $V$ be self-adjoint and relatively form-compact with respect to $T.$
For every $\lambda>0,$ we have
$$\Tr(\chi_{(-\infty,-\lambda)}(T+V)) = \Tr(\chi_{(1,\infty)}(-(T+\lambda)^{-\frac12}V(T+\lambda)^{-\frac12})).$$
\end{theorem}

The following simple lemma is a well-known consequence of the inequality
\begin{equation}\label{perturbation_inequality}
    \Tr(\chi_{(t+s,\infty)}(T+S)) \leq \Tr(\chi_{(t,\infty)}(T))+\Tr(\chi_{(s,\infty)}(S)),\quad t,s>0,\; T=T^*,S=S^* \in \Kc.
\end{equation}
See e.g. \cite[Chapter 9, Theorem 9]{Birman-Solomyak-book}. We present a short argument for convenience.
\begin{lemma}\label{bs approximation lemma} Let $\{A_n\}_{n\geq0}$ be a sequence of self-adjoint compact operators. If $A_n\to A$ in the uniform norm, then
$$\limsup_{n\to\infty}\Tr(\chi_{(1,\infty)}(A_n))\leq \Tr(\chi_{[1,\infty)}(A))$$
and
$$\Tr(\chi_{(1,\infty)}(A))\leq \liminf_{n\to\infty}\Tr(\chi_{(1,\infty)}(A_n)).$$
\end{lemma}
\begin{proof} Fix $0 < \varepsilon<1.$ For all $n\geq 0$ we have
$$\Tr(\chi_{(1,\infty)}(A_n))\leq \Tr(\chi_{(1-\epsilon,\infty)}(A))+\Tr(\chi_{(\epsilon,\infty)}(A_n-A)).$$
Since $A_n-A\to0$ in the uniform norm, it follow that for sufficiently large $n$ we have $\|A_n-A\|_{\infty}<\varepsilon.$ Thus,
$$\Tr(\chi_{(1,\infty)}(A_n))\leq \Tr(\chi_{(1-\epsilon,\infty)}(A))$$
for all sufficiently large $n.$ Passing $n\to\infty,$ we obtain
$$\limsup_{n\to\infty} \Tr(\chi_{(1,\infty)}(A_n))\leq \Tr(\chi_{(1-\epsilon,\infty)}(A)).$$
Taking the infimum over $\epsilon>0,$ we obtain the first inequality.

To see the second inequality, again fix $\varepsilon>0.$ We have
$$\Tr(\chi_{(1+\epsilon,\infty)}(A))\leq \Tr(\chi_{(1,\infty)}(A_n))+\Tr(\chi_{(\epsilon,\infty)}(A-A_n)).$$
Since $A_n-A\to0$ in the uniform norm, it follows that $\|A_n-A\|_{\infty}<\epsilon$ for all sufficiently large $n$. Thus,
$$\Tr(\chi_{(1+\epsilon,\infty)}(A))\leq \Tr(\chi_{(1,\infty)}(A_n))$$
for all sufficiently large $n.$ Passing $n\to\infty,$ we obtain
$$\Tr(\chi_{(1+\epsilon,\infty)}(A))\leq \liminf_{n\to\infty} \Tr(\chi_{(1,\infty)}(A_n)).$$
Taking the supremum over $\epsilon>0,$ we obtain the second inequality.
\end{proof}
Combining this assertion with the Birman-Schwinger principle yields the following:
\begin{corollary}\label{bs corollary} Let $T$ be a self-adjoint positive unbounded linear operator on a Hilbert space $H.$ Let $V$ be a self-adjoint linear operator relatively compact with respect to $T.$ Suppose that
\begin{enumerate}[{\rm (i)}]
\item $T$ has a trivial kernel and $T^{-\frac12}VT^{-\frac12}$ is compact.
\item if $\epsilon\downarrow 0,$ then
$$(T+\epsilon)^{-\frac12}V(T+\epsilon)^{-\frac12}\to T^{-\frac12}VT^{-\frac12}$$
in the uniform norm.
\end{enumerate}
We have
$$\Tr(\chi_{(1,\infty)}(-T^{-\frac12}VT^{-\frac12}))\leq \Tr(\chi_{(-\infty,0)}(T+V))\leq \Tr(\chi_{[1,\infty)}(-T^{-\frac12}VT^{-\frac12})).$$		
\end{corollary}
\begin{proof} By Theorem \ref{bs_principle}, we have
$$\Tr(\chi_{(-\infty,-\frac1n)}(T+V)) = \Tr( \chi_{(1,\infty)}(A_n)),\quad A_n=-(T+\frac1n)^{-\frac12}V(T+\frac1n)^{-\frac12},\quad n\geq 1.$$	
Thus,
$$\Tr(\chi_{(-\infty,0)}(T+V))=\lim_{n\to\infty}\Tr(\chi_{(-\infty,-\frac1n)}(T+V))=\lim_{n\to\infty}\Tr( \chi_{(1,\infty)}(A_n)).$$
Setting
$$A=-T^{-\frac12}VT^{-\frac12}$$
and noting that $A_n\to A$ in the uniform norm, we infer the assertion from Lemma \ref{bs approximation lemma}.
\end{proof}

\subsection{Asymptotic Birman-Schwinger principle}

It follows from Corollary \ref{bs corollary}, that given $h>0$ we may replace $T$ with $hT$ to conclude that
\[
    \Tr(\chi_{(h,\infty)}(-T^{-\frac12}VT^{-\frac12}) \leq \Tr(\chi_{(-\infty,0)}(hT+V)) \leq \Tr(\chi_{[h,\infty)}(-T^{-\frac12}VT^{-\frac12})).
\]
In the event that $T^{-\frac12}VT^{-\frac12} \in \Lc_{q,\infty}$ for some $q>0$ we have
\[
    \lim_{h\to 0} h^{q}\Tr(\chi_{(h,\infty)}(-T^{-\frac12}VT^{-\frac12})) = \lim_{h\to 0} h^{q}\Tr(\chi_{[h,\infty)}(-T^{-\frac12}VT^{-\frac12}))
\]
if either limit exists; this follows from Lemma \ref{trivial spectral lemma}.
It follows that
\[
    \lim_{h\to 0} h^q\Tr(\chi_{(-\infty,0)}(hT+V)) = \lim_{h\to 0} h^q\Tr(\chi_{(h,\infty)}(-T^{-\frac12}VT^{-\frac12}))
\]
whenever the limit on the right hand side exists. 

The following result is not likely to be novel, however it is convenient to give a proof here in the present context.
\begin{theorem}\label{bs asymptotic} Let $T$ be a self-adjoint positive unbounded linear operator on a Hilbert space $H.$ Let $V$ be a self-adjoint linear operator, relatively form compact with respect to $H.$ Suppose that $p>2$ and
\begin{enumerate}[{\rm (i)}]
\item $T$ has a trivial kernel and $V_{\pm}^{\frac12}T^{-\frac12}\in\mathcal{L}_{p,\infty};$
\item $\Im(V_{\pm}^{\frac12}T^{-\frac12}) \in (\mathcal{L}_{p,\infty})_0;$
\end{enumerate}
It follows that
$$\lim_{h\downarrow0}h^{\frac{p}{2}}\Tr(\chi_{(-\infty,0)}(hT+V))=\lim_{t\to\infty}t\mu\Big(t,T^{-\frac12}V_-T^{-\frac12}\Big)^{\frac{p}{2}}$$
provided that the limit on the right hand side exists.
\end{theorem}

\begin{remark}
    Note that
    \[
        [T^{-\frac12},V_{\pm}^{\frac12}] = -2i\Im(V_{\pm}^{\frac12}T^{-\frac12})
    \]
    provided that the composition $T^{-\frac12}V_{\pm}^{\frac12}$ makes sense.    
\end{remark}

By the assumptions in Theorem \ref{bs_asymptotic}, we have $V_{\pm}^{\frac12}T^{-\frac12}\in \Lc_{p,\infty}.$ We shall define the operator $T^{-\frac12}VT^{-\frac12}$ as
\[
    T^{-\frac12}VT^{-\frac12} := (V_+^{\frac12}T^{-\frac12})^*V_+^{\frac12}T^{-\frac12}-(V_-^{\frac12}T^{-\frac12})^*V_-^{\frac12}T^{-\frac12}.
\]
\begin{lemma}\label{lap lemma} Let $T$ and $V$ be as in Theorem \ref{bs asymptotic}. We have
$$(T+\epsilon)^{-\frac12}V(T+\epsilon)^{-\frac12}\to T^{-\frac12}VT^{-\frac12},\quad \epsilon\downarrow0,$$
in the uniform norm.
\end{lemma}
\begin{proof} Denote for brevity
$$X=T^{-\frac12}VT^{-\frac12},\quad A_{\epsilon}=T^{\frac12}(T+\epsilon)^{-\frac12}.$$
By assumption, we have $X\in\mathcal{L}_{\frac{p}{2},\infty}\subset\mathcal{L}_p.$ Clearly, each $A_{\epsilon}$ is a contraction and $A_{\epsilon}\to 1$ in the strong operator topology as $\epsilon\downarrow0.$ Thus, $A_{\epsilon}XA_{\epsilon}\to X$ in $\mathcal{L}_p$ and, therefore, in $\mathcal{L}_{\infty}.$ Since
$$A_{\epsilon}XA_{\epsilon}=(T+\epsilon)^{-\frac12}V(T+\epsilon)^{-\frac12},$$
the assertion follows.
\end{proof}

Recall (see \cite{Davies-jlms-1988}) the Lipschitz inequality for the absolute value function: if $A$ and $B$ are self-adjoint operators, then
$$\||A|-|B|\|_{q,\infty} \lesssim_q \|A-B\|_{q,\infty},\quad 1<q<\infty.$$
The following is a straightforward consequence:
\begin{equation}\label{positive part lipschitz}
\|A_+-B_+\|_{q,\infty}\leq c_q\|A-B\|_{q,\infty},\quad 1<q<\infty.
\end{equation}
Recall that $A_+ = \frac{1}{2}(A+|A|)$ is the positive part of $A,$ and similarly $B_+$ is the positive part of $B.$ This comes with the associated implication that if $A$ and $B$ are self-adjoint operators such that $A-B \in (\Lc_{q,\infty})_0$ then $A_+-B_+ \in (\Lc_{q,\infty})_0.$

\begin{lemma}\label{last section commutator lemma}
Let $T$ and $V$ be as in Theorem \ref{bs asymptotic}. We have
$$(T^{-\frac12}VT^{-\frac12})_- - V_{-}^{\frac12}T^{-1}V_{-}^{\frac12} \in (\mathcal{L}_{\frac{p}{2},\infty})_0.$$
\end{lemma}
\begin{proof} First, we write
$$T^{-\frac12}V_{\pm}T^{-\frac12}-V_{\pm}^{\frac12}T^{-1}V_{\pm}^{\frac12}=[T^{-\frac12},V_{\pm}^{\frac12}]\cdot V_{\pm}^{\frac12}T^{-\frac12}-V_{\pm}^{\frac12}T^{-\frac12}\cdot[T^{-\frac12},V_{\pm}^{\frac12}].$$
The assumptions that $[V_{\pm},T^{-\frac12}] \in (\Lc_{p,\infty})_0$, $V_{\pm}^{\frac12}T^{-\frac12}\in \Lc_{p,\infty}$ and the H\"older inequality yield
$$T^{-\frac12}V_{\pm}T^{-\frac12}-V_{\pm}^{\frac12}T^{-1}V_{\pm}^{\frac12}\in(\mathcal{L}_{\frac{p}{2},\infty})_0.$$   
Consequently,
$$-T^{-\frac12}VT^{-\frac12}-V_{-}^{\frac12}T^{-1}V_{-}^{\frac12}+V_+^{\frac12}T^{-1}V_+^{\frac12}\in(\mathcal{L}_{\frac{p}{2},\infty})_0.$$    
Since $p>2,$ we may apply \eqref{positive part lipschitz} with $q = \frac{p}{2}$ yielding
$$\Big(-T^{-\frac12}VT^{-\frac12}\Big)_+-\Big(V_{-}^{\frac12}T^{-1}V_{-}^{\frac12}-V_+^{\frac12}T^{-1}V_+^{\frac12}\Big)_+\in(\mathcal{L}_{\frac{p}{2},\infty})_0.$$    
By definition, the operators $V_-$ and $V_+$ have orthogonal range. Given that $V_-^{\frac12}T^{-1}V_-^{\frac12}$
and $V_+^{\frac12}T^{-1}V_+^{\frac12}$ are positive operators acting between orthogonal subspaces, it follows that
$$\Big(V_{-}^{\frac12}T^{-1}V_{-}^{\frac12}-V_+^{\frac12}T^{-1}V_+^{\frac12}\Big)_+= V_{-}^{\frac12}T^{-1}V_{-}^{\frac12}.$$
Therefore,
$$\Big(-T^{-\frac12}VT^{-\frac12}\Big)_+-V_{-}^{\frac12}T^{-1}V_{-}^{\frac12}\in(\mathcal{L}_{\frac{p}{2},\infty})_0$$    
and the assertion follows.
\end{proof}

Below we use the standard fact
\begin{lemma}\label{fedor's fix lemma} If $A,B\in\mathcal{L}_{\frac{p}{2},\infty}$ are self-adjoint operators with $A-B\in(\mathcal{L}_{\frac{p}{2},\infty})_0,$ then
$$\liminf_{h\downarrow0}h^{\frac{p}{2}}\Tr(\chi_{(h,\infty)}(B))=\liminf_{h\downarrow0}h^{\frac{p}{2}}\Tr(\chi_{(h,\infty)}(A)).$$
The same equality holds for the upper limits.
\end{lemma}
\begin{proof} Fixing $\varepsilon>0$ and applying \eqref{perturbation_inequality} yields
$$\Tr(\chi_{(h(1+\varepsilon),\infty)}(A)) \leq \Tr(\chi_{(h,\infty)}(B))+\Tr(\chi_{(h\varepsilon,\infty)}(A-B)).$$
Therefore,
$$\Tr(\chi_{(h(1+\varepsilon),\infty)}(A)) \leq \Tr(\chi_{(h,\infty)}(B))+ o(h^{-\frac{p}{2}}),\quad h\downarrow0.$$
Consequently,
$$\liminf_{h\downarrow0}h^{\frac{p}{2}}\Tr(\chi_{(h(1+\varepsilon),\infty)}(A)) \leq \liminf_{h\downarrow0}h^{\frac{p}{2}}\Tr(\chi_{(h,\infty)}(B)).$$
In other words,
$$(1+\epsilon)^{-\frac{p}{2}}\liminf_{h\downarrow0}h^{\frac{p}{2}}\Tr(\chi_{(h,\infty)}(A)) \leq \liminf_{h\downarrow0}h^{\frac{p}{2}}\Tr(\chi_{(h,\infty)}(B)).$$
Sending $\varepsilon\to 0,$ we conclude that
$$\liminf_{h\downarrow0}h^{\frac{p}{2}}\Tr(\chi_{(h,\infty)}(A)) \leq \liminf_{h\downarrow0}h^{\frac{p}{2}}\Tr(\chi_{(h,\infty)}(B)).$$
Swapping $A$ and $B,$ we obtain the opposite inequality. This completes the proof.
\end{proof}

\begin{lemma}\label{last section estimate from below lemma} Let $T$ and $V$ be as in Theorem \ref{bs asymptotic}. We have
$$\liminf_{h\to0}h^{\frac{p}{2}}N(0,hT+V)\geq \liminf_{h\to0}h^{\frac{p}{2}}{\rm Tr}(\chi_{(h,\infty)}(T^{-\frac12}V_-T^{-\frac12})).$$   
\end{lemma}
\begin{proof} By Lemma \ref{lap lemma}, the assumptions in Corollary \ref{bs corollary} hold. Corollary \ref{bs corollary} asserts that
$$N(0,hT+V)\geq \Tr(\chi_{(1,\infty)}(-(hT)^{-\frac12}V(hT)^{-\frac12}))=\Tr(\chi_{(h,\infty)}(-T^{-\frac12}VT^{-\frac12})).$$
Using the identity
$$\Tr(\chi_{(h,\infty)}(-A))=\Tr(\chi_{(h,\infty)}(A_-)),\quad A=A^* \in \Kc(H)$$
we write
$$N(0,hT+V) \geq \Tr(\chi_{(h,\infty)}((T^{-\frac12}VT^{-\frac12})_-)).$$

Setting
$$A=(T^{-\frac12}VT^{-\frac12})_-,\quad B=V_{-}^{\frac12}T^{-1}V_{-}^{\frac12}$$
it follows from Lemma \ref{last section commutator lemma} that $A-B \in (\Lc_{\frac{p}{2},\infty})_0.$ Lemma \ref{fedor's fix lemma} now yields
$$\liminf_{h\to0}h^{\frac{p}{2}}N(0,hT+V)$$
$$\geq\liminf_{h\to0}h^{\frac{p}{2}}\Tr(\chi_{(h,\infty)}((T^{-\frac12}VT^{-\frac12})_-))=\liminf_{h\to0}h^{\frac{p}{2}}{\rm Tr}(\chi_{(h,\infty)}(V_-^{\frac12}T^{-1}V_-^{\frac12})).$$ 
Since
$$\mu\Big(V_-^{\frac12}T^{-1}V_-^{\frac12}\Big)=\mu\Big(T^{-\frac12}V_-T^{-\frac12}\Big),$$
the assertion follows.
\end{proof}

\begin{lemma}\label{last section estimate from above lemma} Let $T$ and $V$ be as in Theorem \ref{bs asymptotic}. We have
$$\limsup_{h\to0}h^{\frac{p}{2}}N(0,hT+V)\leq \limsup_{h\to0}h^{\frac{p}{2}}{\rm Tr}(\chi_{(h,\infty)}(T^{-\frac12}V_-T^{-\frac12})).$$   
\end{lemma}
\begin{proof} By Lemma \ref{lap lemma}, the assumptions in Corollary \ref{bs corollary} hold. Corollary \ref{bs corollary} asserts that
$$N(0,hT+V)\leq \Tr(\chi_{[1,\infty)}(-(hT)^{-\frac12}V(hT)^{-\frac12}))=\Tr(\chi_{[h,\infty)}(-T^{-\frac12}VT^{-\frac12})).$$
Recalling that
$$\Tr(\chi_{[h,\infty)}(B))\leq \Tr(\chi_{[h,\infty)}(A)),\quad B\leq A,$$
and setting
$$A=T^{-\frac12}V_-T^{-\frac12},\quad B=-T^{-\frac12}VT^{-\frac12},$$
we obtain
$$N(0,hT+V)\leq \Tr(\chi_{[h,\infty)}(T^{-\frac12}V_-T^{-\frac12})).$$
Passing $h\to0,$ Lemma \ref{trivial spectral lemma} completes the proof.
\end{proof}

\begin{proof}[Proof of Theorem \ref{bs asymptotic}] 
Combining Lemmas \ref{last section estimate from above lemma} and \ref{last section estimate from below lemma}, we have
\begin{align*}
    \liminf_{h\to 0}\; &h^{\frac{p}{2}}{\rm Tr}(\chi_{(h,\infty)}(T^{-\frac12}V_-T^{-\frac12})) \leq \liminf_{h\to 0} h^{\frac{p}{2}}N(0,hT+V)\\
    &\quad \leq \limsup_{h\to 0} h^{\frac{p}{2}}N(0,hT+V) \leq \limsup_{h\to 0} h^{\frac{p}{2}}{\rm Tr}(\chi_{(h,\infty)}(T^{-\frac12}V_-T^{-\frac12}))
\end{align*}
From Lemma \ref{trivial spectral lemma}, we have
\begin{align*}
    \liminf_{t\to\infty}\;& t\mu\Big(t,T^{-\frac12}V_-T^{-\frac12}\Big)^{\frac{p}{2}} \leq \liminf_{h\to 0} h^{\frac{p}{2}}N(0,hT+V)\\
    &\quad \leq \limsup_{h\to 0} h^{\frac{p}{2}}N(0,hT+V) \leq \limsup_{t\to\infty} t\mu\Big(t,T^{-\frac12}V_-T^{-\frac12}\Big)^{\frac{p}{2}}.
\end{align*}
Hence,
\[
    \lim_{h\to 0} h^{\frac{p}{2}}N(0,hT+V) = \lim_{t\to\infty} t\mu\Big(t,T^{-\frac12}V_-T^{-\frac12}\Big)^{\frac{p}{2}}
\]
if the limit on the right hand side exists.
\end{proof}

\subsection{Semiclassical asymptotics for stratified Lie groups}

Recall from the proof of Lemma \ref{g_Delta_L_p_bounds} the theorem of Christ \cite[Proposition 3.1]{Christ1991} that
for all $g \in L_2(\Rl_+,t^{\frac{d_{\hom}}{2}-1}\,dt)$ we have $g(-\Delta) \in L_2(\VN(G),\tau)$ and there is a non-zero constant $c$ such that
\[
    \|g(-\Delta)\|_{L_2(\VN(G),\tau)} = c\|g\|_{L_2(\mathbb{R}_+,t^{\frac{d_{\hom}}{2}-1}\,dt)}.
\]
A consequence of this identity is that the spectrum of $-\Delta$ is purely absolutely continuous, and in particular $-\Delta$ has trivial kernel.

\begin{lemma}\label{semi-classical verification lemma} Let $G$ be a stratified Lie group with $d_{\hom}>2.$ Let $T=-\Delta$ and let $V=M_f,$ where $f\in L_{\frac{d_{\hom}}{2}}(G)$ is real-valued.
The assumptions in Theorem \ref{bs asymptotic} hold.
\end{lemma}
\begin{proof} 
Note that $V_{\pm}^{\frac12}\in L_{d_{\hom}}(G).$ The first assumption follows from Theorem \ref{main_nontrivial_cwikel_theorem}.

To prove the second assumption, fix a sequence $\{f_k\}_{k\geq0}\subset C^{\infty}_c(G)$ such that $f_k\to f_+^{\frac12}$ in $L_{d_{\hom}}(G).$ As established in the proof of the Lemma \ref{second_weyl_lemma}, we have
$$M_{f_k}(-\Delta)^{-\frac12}-M_f(1-\Delta)^{-\frac12}\in \mathcal{L}_{d_{\hom}}\subset (\mathcal{L}_{d_{\hom},\infty})_0.$$
Thus,
$$[M_{f_k},(-\Delta)^{-\frac12}-(1-\Delta)^{-\frac12}]\in (\mathcal{L}_{d_{\hom},\infty})_0.$$
By Theorem \ref{MSX_commutator_estimate}, we have 
$$[M_{f_k},(1-\Delta)^{-\frac12}]\in \mathcal{L}_{\frac{d_{\hom}}{2},\infty}\subset (\mathcal{L}_{d_{\hom},\infty})_0.$$
Thus,
$$[M_{f_k},(-\Delta)^{-\frac12}]\in (\mathcal{L}_{d_{\hom},\infty})_0.$$

On the other hand, Theorem \ref{main_nontrivial_cwikel_theorem}.\eqref{mncta} implies that
$$[M_{f_k},(-\Delta)^{-\frac12}]\to [M_{f_+}^{\frac12},(-\Delta)^{-\frac12}]$$
in $\mathcal{L}_{d_{\hom},\infty}.$ Hence,
$$[M_{f_+}^{\frac12},(-\Delta)^{-\frac12}]\in(\mathcal{L}_{d_{\hom},\infty})_0.$$
Similarly, one can prove the commutator estimate for $f_-.$ Thus, the second assumption in Theorem \ref{bs asymptotic} also holds.
\end{proof}

\begin{proof}[Proof of Corollary \ref{semi-classical corollary}] By Lemma \ref{semi-classical verification lemma}, Theorem \ref{bs asymptotic} holds for $T=-\Delta$ and $V = M_f$ where $f \in L_{\frac{d_{\hom}}{2}}(G)$ is real-valued, provided that $d_{\hom}>2.$ The assertion follows immediately from that theorem and Theorem \ref{main_asymptotic_formula}.
\end{proof}

\section{Application to Scattering Theory}\label{scattering_section}
Finally we indicate the use of the $\Lc_1$-Cwikel estimate of Theorem \ref{main_nontrivial_cwikel_theorem} in potential scattering theory. This line of argument is quite standard and closely follows \cite[Chapter 6]{Simon-trace-ideals-2005}, however we have included it in order to illustrate the benefit of proving Cwikel-type estimates in new setting.

If $A$ and $B$ are self-adjoint operators on a Hilbert space $H,$ where $A$ has purely absolutely continuous spectrum, then the wave operators $\Omega_{\pm}(A,B)$ are defined as
\[
    \Omega_{\pm}(A,B) := \mathrm{s}-\lim_{t\to \pm\infty} e^{-itB}e^{itA}.
\]
Here, $\mathrm{s}-\lim$ denotes the limit in the strong operator topology. The wave operators $\Omega_{\pm}(A,B)$
are said to be complete if the range of $\Omega_{\pm}(A,B)$ is exactly $H.$ Birman's theorem \cite[Corollary 6.7]{Simon-trace-ideals-2005} asserts that for $\Omega_{\pm}(A,B)$ to exist and be complete it suffices that
\[
    \dom(|A|^{1/2}) = \dom(|B|^{1/2})
\]
(i.e., $A$ and $B$ have the same form domain), and that for all $a>0$ we have
\[
    \chi_{(-a,a)}(A)(A-B)\chi_{(-a,a)}(B)\in \Lc_1.
\]
Therefore, a sufficient condition is that there exists $N>0$ such that
\[
    (A+i)^{-N}(A-B)(B+i)^{-N} \in \Lc_1.
\]
In the traditional setting of scattering theory, we consider $A = -\Delta,$ the usual Laplace operator on $\mathbb{R}^d,$ 
and $B = -\Delta+V,$ where $V = M_f$ is an operator of pointwise mutliplication by $f \in \ell_1(L_2)(\mathbb{R}^d).$ In this
case the classical estimates for $M_f(1-\Delta)^{-N}$ due to Birman and Solomyak suffice to deduce the existence and completeness of the wave operators,
see \cite[Chapter 6]{Simon-trace-ideals-2005}.

The following simple lemma suffices for our purposes.
\begin{lemma}\label{abstract birman verification lemma} Let $A$ and $B$ be self-adjoint operators with equal form domain, and let $V := A-B$ be their difference. If
\begin{enumerate}[{\rm (i)}]
\item there exists $p\in\mathbb{N}$ such that
$$V(A+i)^{-1}\in\mathcal{L}_p;$$
\item there exists $N\in\mathbb{N}$ such that
$$(A+i)^{-N}V(A+i)^{-N}\in\mathcal{L}_1,\quad V(A+i)^{-N},(A+i)^{-N}V\in\mathcal{L}_2;$$
\end{enumerate}
then
$$(A+i)^{-MN}V(B+i)^{-MN}\in\mathcal{L}_{\frac{p}{M}},\quad 1\leq M\leq p.$$
\end{lemma}
\begin{proof} 
We suppose $p\geq 2$ (otherwise there is nothing to prove). We prove the assertion by induction on $M.$ The $M=1$ case is trivial. Suppose the assertion is holds for $M-1,$ where $1\leq M-1<p$ and let us prove it for $M.$ 
Clearly,
$$(A+i)^{-MN}V(B+i)^{-MN}=(A+i)^{-MN}V(A+i)^{-MN}+(A+i)^{-MN}V\cdot\Big((B+i)^{-MN}-(A+i)^{-MN}\Big).$$
By the resolvent identity, we have
$$(B+i)^{-MN}-(A+i)^{-MN}=-\sum_{k=0}^{MN-1}(B+i)^{-k-1}V(A+i)^{k-MN}.$$
Therefore
\begin{align}
    (A+i)^{-MN}V(B+i)^{-MN}&=(A+i)^{-MN}V(A+i)^{-MN}\nonumber\\
                           &\quad-\sum_{k=0}^{MN-1}(A+i)^{-MN}V(B+i)^{-k-1}V(A+i)^{k-MN}\label{iterated_resolvent_identity}.
\end{align}
The first summand belongs to $\Lc_{\frac{p}{M}},$ because
\begin{align*}
    (A+i)^{-MN}&V(A+i)^{-MN}=(A+i)^{-(M-1)N}\cdot (A+i)^{-N}V(A+i)^{-N}\cdot (A+i)^{-(M-1)N}\\
               &\in\mathcal{L}_{\infty}\cdot\mathcal{L}_1\cdot\mathcal{L}_{\infty}=\mathcal{L}_1\subseteq\mathcal{L}_{\frac{p}{M}}.
\end{align*}
To see that the sum over $k$ on the right hand side of \eqref{iterated_resolvent_identity} belongs to the same ideal,
first we consider the case $0\leq k\leq (M-1)N.$ We have
\begin{align*}
    (A+i)^{-MN}&V(B+i)^{-k-1}V(A+i)^{k-MN}\\
               &=(A+i)^{-(M-1)N}\cdot (A+i)^{-N}V\cdot (B+i)^{-k-1}\cdot V(A+i)^{-N}\cdot (A+i)^{k-(M-1)N}\\
               &\in\mathcal{L}_{\infty}\cdot\mathcal{L}_2\cdot\mathcal{L}_{\infty}\cdot\mathcal{L}_2\cdot\mathcal{L}_{\infty}\\
               &=\mathcal{L}_1\subseteq\mathcal{L}_{\frac{p}{M}}.
\end{align*}
Next, we consider $(M-1)N+1\leq k\leq MN-1.$ In this case we rewrite the $k$th summand on the right hand side of \eqref{iterated_resolvent_identity}
as
\begin{align*}
    &(A+i)^{-MN}V(B+i)^{-k-1}V(A+i)^{k-MN}\\
    &=(A+i)^{-N} (A+i)^{-(M-1)N}V(B+i)^{-(M-1)N} (B+i)^{(M-1)N-k-1} V(A+i)^{-1} (A+i)^{k+1-MN}.
\end{align*}
Since $(A+i)^{-(M-1)N}V(B+i)^{-(M-1)N} \in \Lc_{\frac{p}{M-1}}$ by the inductive hypothesis, it follows from H\"{o}lder's inequality that
$$(A+i)^{-MN}V(B+i)^{-k-1}V(A+i)^{k-MN}\in\mathcal{L}_{\infty}\cdot\mathcal{L}_{\frac{p}{M-1}}\cdot\mathcal{L}_{\infty}\cdot\mathcal{L}_p\cdot\mathcal{L}_{\infty}=\mathcal{L}_{\frac{p}{M}}.$$
Hence, every term on the right hand side of \eqref{iterated_resolvent_identity} belongs to $\Lc_{\frac{p}{M}},$ and therefore.
$$(A+i)^{-MN}V(B+i)^{-MN}\in\mathcal{L}_{\frac{p}{M}}.$$
This completes the induction, and hence we deduce the result.
\end{proof}

\begin{theorem} Let $G$ be a stratified Lie group with $d_{\hom}>4,$ and let $f\in(L_{\frac{d_{\hom}}{2}}\cap L_1)(G)$ be real valued. The wave operators exist and are complete for the couple $(-\Delta,-\Delta+M_f),$ where $-\Delta+M_f$ is understood as a quadratic form sum.
\end{theorem}
\begin{proof} 
We use the notations $A=-\Delta,$ $B=-\Delta+M_f$ and $V=M_f.$ It follows immediately from Lemma \ref{christ_restatement} that spectrum of $A$ is absolutely continuous.

By Theorem \ref{main_nontrivial_cwikel_theorem}, we have
$$\|V(A+i)^{-1}\|_{\frac{d_{\hom}}{2},\infty}\leq c_G\|f\|_{\frac{d_{\hom}}{2}}.$$
Let $p\in\mathbb{N}$ be such that $p>2$ and $p>\frac{d_{\hom}}{2}.$ It follows that $V(A+i)^{-1}\in\mathcal{L}_p.$ 
The assumption on $f$ implies that in particular $f\in L_2(G).$ By Theorem \ref{cwikel schatten_theorem}, we have
$$V(A+i)^{-p}, (A+i)^{-p}V\in\mathcal{L}_2.$$
Due to the assumption that $f\in L_1(G),$ Theorem \ref{cwikel schatten_theorem} yields
$$(A+i)^{-p}V(A+i)^{-p}\in\mathcal{L}_1.$$
By Lemma \ref{abstract birman verification lemma} it follows that
$$(A+i)^{-p^2}V(B+i)^{-p^2}\in\mathcal{L}_1.$$
Multiplying on the left by $\chi_{(-a,a)}(A)$ and on the right by $\chi_{(-a,a)}(B),$ we obtain 
$$\chi_{(-a,a)}(A)(A-B)\chi_{(-a,a)}(B)\in \Lc_1.$$
The assertion follows now from Birman's theorem.
\end{proof}

\begin{remark}
    Since the spectrum of $-\Delta$ is purely absolutely continuous, it follows
    that the scattering matrix of the pair $(-\Delta,-\Delta+M_f)$ is unitary
    on $L_2(G).$
\end{remark}

\end{document}